\documentclass[11pt]{article}
\usepackage[margin=.9in]{geometry}

\usepackage{graphicx}
\usepackage{subcaption}
\usepackage{caption}
\usepackage{amssymb}
\usepackage{amsmath}
\usepackage{verbatim,booktabs}
\usepackage{longtable}
\usepackage{float}
\usepackage{mathrsfs}
\usepackage{threeparttable}
\usepackage{multirow}
\usepackage[usenames,dvipsnames]{pstricks}
\usepackage{epsfig}
\usepackage{enumitem}
\usepackage{hyperref}
\usepackage{todonotes}
\usepackage{algpseudocode}

\usepackage[capitalize,noabbrev]{cleveref}

\usepackage[utf8]{inputenc}
\usepackage{fullpage}
\usepackage{amsfonts}
\usepackage{xcolor}
\usepackage{changepage}

\allowdisplaybreaks

\crefname{claim}{Claim}{Claims}
\floatname{algorithm}{Algorithm}

\usepackage{amsthm}

\newtheorem{proposition}{Proposition}

\newtheorem{claim}{Claim}

\def\R{{\mathbb R}}

\def\ie{{i.e.,} }

\def\G{{\mathcal G}}

\def\X{{\mathcal X}}

\def\01{\ensuremath{0\mathord{-}1}}

\def\max{{\rm max}}
\def\min{{\rm min}}

\newcommand{\Tr}{\text{Tr}}

\newcommand{\norm}[1]{\Big\lVert#1\Big\rVert}
\usepackage[ruled,vlined]{algorithm2e}

\title{A scalable linear programming-based framework for data clustering
\thanks{The authors were partially funded by AFOSR grant FA9550-23-1-0123.}}

\author{
Aida Khajavirad
\thanks{Department of Industrial and Systems Engineering, Lehigh University, Bethlehem, PA 18015, USA.
             E-mail: {\tt aida@lehigh.edu}.
             }
             \and
Huanwen Shen
\thanks{Mitch Daniels School of Business, Purdue University, West Lafayette, IN 47907, USA.
             E-mail: {\tt shen809@purdue.edu}.
             }
\and             
Yakun Wang
\thanks{Department of Industrial and Systems Engineering, Lehigh University, Bethlehem, PA 18015, USA.
             E-mail: {\tt yaw220@lehigh.edu}.
             }
}
\date{}

\begin{document}
\maketitle

\begin{abstract}
We extend the linear programming-based algorithm of De Rosa et al~\cite{derKhaWan24} for K-means clustering to two important clustering paradigms: fair K-means clustering and spectral clustering. For fair K-means clustering, we show that widely used notions of group fairness can be incorporated into the partition-matrix formulation of K-means clustering through a linear number of linear inequalities. For spectral clustering, we consider a linear programming relaxation of the minimum ratio-cut problem that fits naturally within the same framework. We complement these formulations with problem-specific initialization and rounding procedures and evaluate the resulting algorithms on a large collection of real-world data sets. Denoting by $n$ the number of data points, our computational results demonstrate that the proposed approach solves $90\%$ of benchmark instances with $n \leq 3000$ to within $1\%$ optimality in at most three hours. This in turn demonstrates the remarkable strength of the proposed LP relaxations in both applications. Moreover, for more than $56\%$ of the instances, the proposed algorithm finds better solutions than those produced by popular fair Lloyd-type and spectral clustering heuristics.
\end{abstract}


\section{Introduction}
\label{sec:intro}

Clustering is a fundamental task in unsupervised learning whose goal is to group data points into subsets, called clusters, so that points within the same cluster are more similar to each other
than points in different clusters. One of the most widely used clustering methods is K-means clustering, which aims to partition a data set into $K$ clusters so that the total variance within the clusters is minimized. More formally, let $\{x^i\}_{i=1}^n$ denote a set of $n$ data points in $\R^m$, and denote by $K$ the number of desired clusters. A partition of $[n] := \{1,\ldots, n\}$ is a family $\{\Gamma_k\}_{k =1}^K$ of non-empty subsets of $[n]$ such that
$\Gamma_a \cap \Gamma_b = \emptyset$ for all $a\neq b \in [K]$ and $\cup_{k \in [K]} {\Gamma_k} = [n]$. The K-means clustering problem can be formulated as a combinatorial optimization problem:
\begin{align}
\label{Kmeans1}
\min  \quad & \sum_{k=1}^K{\sum_{i \in \Gamma_k} { \norm{x^i - \frac{1}{|\Gamma_k|}\sum_{j \in \Gamma_k}{x^j}}_2^2}}\\
\text{s.t.} \quad & \{\Gamma_k\}_{k=1}^{K} \; {\rm is \; a \; partition \; of \; [n]}.  \nonumber
\end{align}
K-means clustering is NP-hard even when there are only two clusters~\cite{Aloise09} or when the data points are in $\R^2$~\cite{MahNimVar09}. The most popular methods for solving K-means clustering are heuristics such as Lloyd’s algorithm~\cite{Lloyd82}, approximation algorithms~\cite{Kan02,FriRezSal19}, and convex relaxations~\cite{PenXia05,PenWei07,Awaetal15,IduMixPetVil17,LiLiLi20,AntoAida20,derKhaWan24}. In this paper, we are interested in solving clustering problems using convex relaxations.
While Problem~\eqref{Kmeans1} is perhaps the most natural formulation for K-means clustering, in the following, we present an alternative formulation for this problem which makes it amenable to various convexification techniques.
Consider a partition $\{\Gamma_k\}_{k =1}^K$ of $[n]$; let ${\bf 1}_{\Gamma_k}$, $k \in [K]$, denote the indicator vector of
the $k$th cluster; \ie  the $i$th component of ${\bf 1}_{\Gamma_k}$ is defined as $({\bf 1}_{\Gamma_k})_i = 1$ if $i \in \Gamma_k$ and $({\bf 1}_{\Gamma_k})_i = 0$ otherwise. Define
the associated \emph{partition matrix} by:
\begin{equation}\label{pm}
X = \sum_{k=1}^K{\frac{1}{|\Gamma_k|}{\bf 1}_{\Gamma_k}{\bf 1}^T_{\Gamma_k}}.
\end{equation}
Define $d_{ij} := ||x^i - x^j||_2^2$ for all $i,j \in [n]$.
Then Problem~\eqref{Kmeans1} can be equivalently written as (see~\cite{LiLiLi20} for the derivation):
\begin{align}
\label{Kmeans2}
\min \quad & \sum_{i,j \in [n]} {d_{ij} X_{ij}}\\
\text{s.t.} \quad & X \; {\rm is \; a \; partition\; matrix \; defined \; by~\eqref{pm}}.  \nonumber
\end{align}
The two prominent types of convex relaxations for K-means clustering are semidefinite programming (SDP) relaxations~\cite{PenWei07} and linear programming (LP) relaxations~\cite{AntoAida20}. In this paper, we are interested in solving clustering problems using LP relaxations. To this end, we next present an LP relaxation of Problem~\eqref{Kmeans2} introduced in~\cite{AntoAida20}. 
Fix a parameter $t \in \{2,\ldots, K\}$;
then an LP relaxation for K-means clustering is given by:
\begin{align}
\label{lp:pK}
\tag{LPK$_t$}
\min  \quad  & \sum_{i,j \in [n]} {d_{ij} X_{ij}}\nonumber\\
\text{s.t.} \quad  & \Tr(X) = K,\quad \sum_{j=1}^n X_{ij} = 1, \quad \forall i \in [n],\nonumber\\
& \sum_{j \in S} {X_{ij}} \leq X_{ii} + \sum_{j, k \in S: j < k} {X_{jk}}, \quad \forall i \in [n], \; \forall S \subseteq [n] \setminus\{i\}: 2 \leq |S| \leq t,\label{eq3k}\\
& X_{ij} \geq 0, \quad X_{ij}=X_{ji}, \quad \forall 1 \leq i < j \leq n \nonumber.
\end{align}
Notice that system~\eqref{eq3k} contains $\Theta(n^{t+1})$ inequalities, which makes the LP prohibitively expensive to solve for large data sets even for $t=2$. To address the  scalability of this LP, the authors 
of~\cite{derKhaWan24} devise a cutting-plane algorithm, which relies on an efficient separation of inequalities~\eqref{eq3k}, lower bounding and upper bounding techniques, and a GPU implementation of PDLP, a first-order primal-dual LP solver~\cite{lu2023cupdlpC}. They then solved Problem~\eqref{lp:pK} for real-world instances with up to $4000$ data points in less than two and a half hours. 
Surprisingly, their numerical experiments with real-world data sets indicate that the LP relaxation is almost always tight; \ie its optimal solution is a partition matrix. 
Motivated by these promising computational results, in this paper, we extend this framework to solve other important clustering formulations; namely, \emph{fair} K-means clustering and \emph{spectral} clustering.

\subsection{Fairness in clustering}

Conventional clustering algorithms group data points solely according to a chosen similarity measure. Consequently, the resulting clusters can exhibit unintended bias with respect to protected attributes such as race or gender. To address this issue, the fair clustering literature has introduced several notions of fairness, each capturing a different formal interpretation of what it means for clusters to treat protected groups equitably. In this paper, we focus on a notion of \emph{group fairness} known as the \emph{disparate impact}~\cite{feldman15}: the idea that the representation of protected groups within each cluster should not differ significantly from their representation in the overall population.  As we detail next, various fair clustering formulations in the literature can be seen as different ways of controlling the deviation from proportional representation.

In~\cite{chier17}, the authors introduce the notion of \emph{balance} for fairness in clustering.  Consider the set of data points $\X:=\{x^i\}_{i=1}^n$ and suppose that each point $x^i$ is associated with a protected attribute $g \in \G$, where $\G$ denotes the set of all protected groups. For each $g \in \G$, denote by $\X_g \subseteq \X$ the set of points with protected attribute $g$.  The balance of the data set $\X$ is then defined as 
\begin{equation}\label{balance}
{\rm balance(\X)} =\frac{\min_{g \in \G} |\X_g|}{\max_{g \in \G}|\X_g|}.
\end{equation}
Given a clustering $\{\Gamma_k\}_{k=1}^K$, the balance of a cluster $\Gamma_k$ for some $k \in [K]$ is defined as:
$$
{\rm balance}(\Gamma_k) =\frac{\min_{g \in \G} |\Gamma_k \cap \X_g|}{\max_{g \in \G}|\Gamma_k \cap \X_g |}.
$$
The balance of a clustering $\{\Gamma_k\}_{k=1}^K$ is then given by:
$$
{\rm balance}(\{\Gamma_k\}_{k=1}^K) = \min_{k \in [K]}{\rm balance}(\Gamma_k).
$$
Clearly, the balance of any clustering is upper bounded by the balance of the data itself. To achieve fairness, one aims to obtain a clustering with higher balance.
Given $t \in (0,1]$, we say that a clustering $\{\Gamma_k\}_{k=1}^K$ is~\emph{$t$-balanced} if 
\begin{equation}\label{balcond}
{\rm balance}(\{\Gamma_k\}_{k=1}^K) \geq t.
\end{equation}
 In~\cite{bera19}, the authors introduce a similar notion of fair clustering by controlling the maximum over-representation and the minimum under-representation of any protected group in any cluster. Namely, given parameters $\alpha, \beta \in [0,1]^q$, where $q:=|\G|$, a clustering $\{\Gamma_k\}_{k=1}^K$ is fair if it satisfies the following inequalities:
\begin{equation}\label{faircond}
   \beta_g\leq\frac{|\Gamma_k \cap \X_g|}{|\Gamma_k|} \leq \alpha_g, \quad \forall g \in \G, \; \forall k \in [K],
\end{equation}
where, without loss of generality, we assume that for each $g \in \G$ we have
$$\beta_g \geq 1-\sum_{h \in \G \setminus \{g\}}{\alpha_h}, \qquad \alpha_g \leq 1-\sum_{h \in \G \setminus \{g\}}{\beta_h}.$$ 
Define 
$$
\beta_{\min} := \min_{g \in \G}\beta_g, \quad \alpha_{\max} := \max_{g \in \G} \alpha_g.
$$
We deduce that if a clustering satisfies inequalities~\eqref{faircond}, then it is $t$-balanced with $t=\frac{\beta_{\min}}{\alpha_{\max}}$. To see this, observe that by~\eqref{faircond} for any $k \in [K]$, we have: 
$$
{\rm balance}(\Gamma_k) =\frac{\min_{g \in \G} |\Gamma_k \cap \X_g|}{\max_{g \in \G}|\Gamma_k \cap \X_g |} \geq  \frac{|\Gamma_k|\cdot \beta_{\min}}{|\Gamma_k|\cdot \alpha_{\max}} = \frac{\beta_{\min}}{\alpha_{\max}}.
$$
In~\cite{gupta23}, the authors introduce another notion of group fairness, called $\tau$-ratio fairness, which ensures that each cluster contains at least a predefined fraction of points from each protected group. Namely, given parameters $\tau_g \in (0,\frac{1}{K}]$ for all $g \in \G$, a clustering $\{\Gamma_k\}_{k=1}^K$ is $\tau$-ratio fair if it satisfies the following inequalities:
\begin{equation}\label{taufair}
    \frac{|\Gamma_k \cap \X_g|}{|\X_g|} \geq \tau_g, \quad \forall g \in \G, \; \forall k \in [K].
\end{equation}
If a clustering $\{\Gamma_k\}_{k=1}^K$ is $\tau$-ratio fair, then it is also $t$-balanced with 
\begin{equation}\label{tval}
t=\frac{\min_{g\in \G} {\tau_g |\X_g|}}{\max_{g \in \G}{(1-(K-1)\tau_g)|\X_g|}}.
\end{equation}
To see this, observe that from~\eqref{taufair} it follows that $$\tau_g|\X_g| \leq |\Gamma_k \cap \X_g|\leq (1-(K-1)\tau_g)|\X_g|, \quad \forall k \in [K].
$$ 
In the special case where $\tau_g = \frac{1}{K}$ for all $g \in \G$, substituting in~\eqref{tval} we deduce that
$${\rm balance}(\{\Gamma_k\}_{k=1}^K) =  {\rm balance}(\X).$$
Finally, in~\cite{lawgun24}, the authors introduce the notion of minimum representation for fairness in clustering. Namely, instead of requiring bounded representations across all clusters, they require that each protected group attains a minimum level of representation in at least a specified number of clusters. 

In this paper, we consider the problem of fair K-means clustering, where fairness is measured using the notion of balance and is enforced through inequalities~\eqref{faircond} or  inequalities~\eqref{taufair}. In addition to their widespread use in the fair clustering literature, as we detail in the next section, these inequalities can be formulated as linear inequalities in terms of partition matrices and hence can be readily incorporated into Problem~\eqref{lp:pK}.

Existing methods for fair clustering are either approximation algorithms~\cite{chier17,bera19,gupta23} or are variants of the popular Lloyd algorithm that somehow incorporate fairness into the clustering heuristic~\cite{ghadiri21,lawgun24}.
Roughly speaking, the approximation algorithms first find cluster centers by solving the clustering problem without fairness constraints, and then find a fair assignment of the data points to these cluster centers. In contrast, in this paper, we propose an LP-based algorithm with performance guaranties to solve the fair clustering problem. The proposed algorithm relies on a strong LP relaxation for fair K-means clustering obtained by incorporating fairness constraints into the LP relaxation of~\cite{derKhaWan24} together with an efficient rounding technique that can be considered as a fair Lloyd-type algorithm first introduced in~\cite{lawgun24}. Our computational results on various real-world data sets with $n \leq 3000$ indicate that about $90\%$ of the instances reach a relative optimality gap of less than $1\%$ within three hours; in fact, more than $72\%$ of the instances reach a relative optimality gap of less than $1\%$ within $500$ seconds. Moreover, for more than $45\%$ of the instances the proposed algorithm finds better solutions than a fair Lloyd-type heuristic.

\subsection{Spectral clustering}

Spectral clustering is a graph-based clustering method that exploits the spectral properties of a similarity graph to capture the intrinsic geometry of the data. Unlike K-means clustering, spectral clustering can successfully identify nonconvex or highly anisotropic clusters. For a comprehensive survey of spectral clustering and its theoretical foundations, we refer the reader to~\cite{Lux07}. In~\cite{LinStr20}, the authors show that spectral clustering can be interpreted as a continuous relaxation of the NP-hard minimum ratio-cut problem~\cite{WagWag93}. We briefly review this connection below.

Let $G$ be a weighted undirected graph with node set $V=\{v_1, \cdots, v_n\}$. Denote by $W=(w_{ij})_{i,j\in [n]}$ the weighted adjacency matrix of $G$. Notice that $W$ is a symmetric matrix with $w_{ii} = 0$ for all $i \in [n]$. Define the weighted degree of node $v_i$, $i \in [n]$, as ${\rm deg}(v_i) = \sum_{j \in [n]}{w_{ij}}$, and define the degree matrix $D$ as the diagonal matrix  whose diagonal entries are the degrees ${\rm deg}(v_1), \cdots, {\rm deg}(v_n)$. The \emph{graph Laplacian} of $G$ is defined as:
\begin{equation}\label{Laplacian}
L:=D-W.
\end{equation}
The minimum ratio-cut problem can then be formulated as (see~\cite{LinStr20} for the derivation):
\begin{align}
\label{minratio}
\min \quad & \sum_{i,j \in [n]} {L_{ij} X_{ij}}\\
\text{s.t.} \quad & X \; {\rm is \; a \; partition\; matrix \; defined \; by~\eqref{pm}},  \nonumber
\end{align}
where $L_{ij}$ denotes the $(i,j)$th entry of the graph Laplacian $L$.
Consider a set of points $\X =\{x^i\}_{i=1}^n$ and denote by $\{\Gamma_k\}_{k=1}^K$ a partition of $\X$ into $K$ clusters.
Define the indicator matrix $U \in \R^{n \times K}$, where the $k$th column of $U$, denoted by $U_k$, is defined as:
\begin{equation}\label{indicator}
U_k= \frac{1}{\sqrt{|\Gamma_k|}} {\bf 1}_{\Gamma_k},
\end{equation}
where ${\bf 1}_{\Gamma_k}$ is the indicator vector of
the $k$th cluster as defined in~\eqref{pm}. It can be checked that
$X= UU^\top$ and therefore Problem~\eqref{minratio}
can be equivalently written as:
\begin{align}\label{minratio2}
\min \quad & \Tr(U^T L U) \\
\text{s.t.} \quad & U \; \text{is an indicator matrix defined by~\eqref{indicator}.} \nonumber 
\end{align}
Spectral clustering is a two-step relax-and-round algorithm for solving Problem~\eqref{minratio2}. In the first step, spectral clustering relaxes the combinatorial constraint on $U$ and solves the following tractable relaxation of Problem~\eqref{minratio2}:
\begin{align}
\min \quad & \Tr(U^T L U) \\
\text{s.t.} \quad & U^\top U = I_K. \nonumber 
\end{align}
It can be shown that the optimal solution $\tilde U$ of the above problem consists of the eigenvectors of $L$ corresponding to its $K$ smallest eigenvalues. Notice that the resulting $\tilde U$ is not an indicator matrix. To obtain a feasible solution of Problem~\eqref{minratio2}, in the second step, spectral clustering performs a rounding step, in which K-means clustering is then applied on the rows of $\tilde U$ using Lloyd's algorithm.
In this way, the $i$th row of $\tilde U$ is interpreted as an embedding of the data point $x^i$ in $\mathbb{R}^K$. 
This two-step procedure is computationally efficient and often produces high-quality clusterings in practice. Nevertheless, spectral clustering remains a heuristic method, and its output does not generally coincide with an optimal solution of Problem~\eqref{minratio2} (or equivalently, Problem~\eqref{minratio}). In~\cite{LinStr20}, the authors propose an SDP relaxation of Problem~\eqref{minratio} and derive sufficient conditions under which this relaxation is tight, \ie its solution is a partition matrix.

Comparing Problem~\eqref{Kmeans2} and Problem~\eqref{minratio2}, we observe that the feasible regions of the two optimization problems are identical.  Consequently, replacing $d_{ij}$ with $L_{ij}$ in the objective function of Problem~\eqref{lp:pK} yields an LP relaxation of Problem~\eqref{minratio}. In this paper, we investigate the numerical properties of this LP relaxation. Namely, we propose an LP-based algorithm with performance guaranties to solve the minimum ratio-cut and hence the spectral clustering problem. We investigate the effectiveness of the proposed algorithm by performing social network analysis using real-world data sets. Our experiments on problems with $n \leq 1500$ indicate that $94\%$ of the instances reach an optimality gap below $1.0\%$ in three hours. In addition, more than $80\%$ of the instances reach an optimality gap of less than $1.0\%$ in about an hour. Interestingly, for $77\%$ if the instances, the proposed algorithm finds a better solution that the popular spectral heuristic.

\paragraph{Organization}
The remainder of this paper is structured as follows. Section~\ref{sec:cutting_plane} reviews the cutting-plane algorithm of~\cite{derKhaWan24} to solve Problem~\eqref{lp:pK}. In Section~\ref{sec:fair_clustering}, we propose an algorithm to solve the fair K-means clustering and perform extensive numerical experiments with real-world data sets. In  Section~\ref{sec:spectral_clustering}, we propose an algorithm to solve spectral clustering and investigate its computational properties by performing community detection using real-world data sets. Further computational results for fair K-means clustering are reported in the Appendix.

\section{A scalable algorithm for K-means clustering}
\label{sec:cutting_plane}

In this section, we provide a brief overview of the cutting-plane algorithm proposed in~\cite{derKhaWan24} to solve Problem~\eqref{lp:pK}. The main obstacle to solving Problem~\eqref{lp:pK} efficiently is that the system~\eqref{eq3k} consists of $\Theta(n^{t+1})$ inequalities. Indeed, as detailed in~\cite{derKhaWan24}, even when $t=2$, for $n > 400$, state-of-the-art LP solvers are unable to solve Problem~\eqref{lp:pK} within four hours. However, the customized algorithm of~\cite{derKhaWan24}, which relies on efficient separation of the inequalities~\eqref{eq3k}, is able to solve Problem~\eqref{lp:pK} for instances with up to $n = 4000$ within two hours. The main components of this algorithm are as follows:
\begin{itemize}[leftmargin=.2in]
    \item [$(i)$] {\bf  Initialization:} {\tt k-means++}, \ie an enhanced implementation of Lloyd's algorithm,  is used to compute an initial feasible solution and hence an upper bound on the optimal value of Problem~\eqref{lp:pK}. Moreover, to construct the first LP, inequalities~\eqref{eq3k} with $t=2$ that are satisfied tightly at the {\tt k-means++}'s solution are selected.
    
    \item [$(ii)$] {\bf Safe lower bounds:} In the first few iterations of the cutting-plane algorithm, the LPs are not solved to optimality and the solver is terminated early. The solutions returned in such cases by the solver are often infeasible. The authors make use of the existing techniques using LP duality to generate valid lower bounds on the optimal value of these intermediate LPs~\cite{neumaier2004safe}. 
    
    \item [$(iii)$] {\bf Rounding:} If at any iteration of the algorithm, the optimal solution of the LP relaxation is not a partition matrix, a rounding scheme proposed in~\cite{PenWei07} is used to ``round'' this solution and obtain a partition matrix whose cost may serve as a good upper bound on the optimal clustering cost. 

    \item [$(iv)$] {\bf Separation:} Fix $t \in \{2,\ldots, K\}$ and let $\tilde X$ denote the solution to the current LP.  For each $i \in [n]$, the separation problem is 
 to find a nonempty subset $S\subseteq[n] \setminus \{i\}$ with $2 \leq |S| \leq t$ such that
$$
w_i(S) := \sum_{j \in S} \tilde X_{i j} -\sum_{j, k \in S: j<k} \tilde X_{j k}>\tilde X_{i i},
$$
or to prove that no such subset exists. Repeated calculations can be avoided by using the relation
$
w_i(S\cup\{k\}) = w_i(S) + \tilde X_{ik}-\sum_{v\in S}\tilde X_{vk}
$,
where $\tilde X_{ij} = \tilde X_{ji}$ for all $i < j$. Enumerating all possible choices for $S$ is too expensive for large-scale problems. Therefore, to construct such a set $S$, the authors use a greedy strategy proposed in~\cite{marzi2019computational} to separate clique inequalities. This separation algorithm is not exact, as it may fail to identify violated inequalities corresponding to some subsets $S$. 
The outline of the separation scheme is provided in Algorithm~\ref{alg:sep}.
\end{itemize}

\begin{algorithm}[htb]
\SetKwFunction{sep}{separate}
\caption{The heuristic algorithm for separating inequalities~\eqref{eq3k}}
\label{alg:sep}
\KwIn{LP solution $\tilde X$, violation tolerance $\epsilon_{\rm vio}$, and $t_{\max}$}
\KwOut{A number of inequalities of the form~\eqref{eq3k} violated at $\tilde X$.}
            \For{$i \in [n]$ \textbf{in parallel}}{
                \For{$j \in [n] \setminus \{i\}$}{          
                    Initialize $S = \{j\}$, $w_i(S) = \tilde X_{ij}$, and $c = j$  \\             
                    \While{$|S| < t_{\max}$}
                        {select $k \in \{c + 1, \dots, n\} \setminus (S \cup \{i\})$ that maximizes $\gamma_i(k) = \tilde X_{ik}-\sum_{l\in S,l<k} \tilde X_{lk}$.\\
                        
                         Update $\bar{S} = S\cup\{k\} $, and $c = k$ \\
                        Compute $w_i(\bar{S}) = w_i(S) +\gamma_i(k)$. \\                      
                        \If{$ w_i(\bar{S}) > \tilde X_{ii} + \epsilon_{\rm vio} $}{
                        Add $ \sum_{k \in \bar S} {X_{ik}} \leq X_{ii} + \sum_{l, k \in \bar S: l < k} {X_{lk}}$
                        to the set of violated inequalities.
                        }
                        Update $S = \bar{S}$
                    }
                }
            }
\end{algorithm}

\begin{algorithm}[htb]
\SetKwFunction{cut}{Iterative LP solver}
\KwIn{Data points $\{x^i\}_{i=1}^n$, number of clusters $K$, optimality tolerance $\epsilon_{\rm opt}$, initial number of inequalities $p_{\rm init}$, number of inequalities added at each round of cut generation $p_{\max}$, and the solver time limit $T$ for each intermediate LP.}
\KwOut{Partition matrix $X_{ub}$ and optimality gap $r_g$.}
\textbf{Initialize:} 
    Set the lower bound $f_{lb} = -\infty$ and the optimality gap $r_g = +\infty$.
    Run {\tt k-means++} to get a partition matrix $X_{ub}$.
    Set the upper bound $f_{ub}$ as the cost of $X_{ub}$.    
    Randomly select at most $p_{\rm init}$ of inequalities~\eqref{eq3k} with $t = 2$ that are active at $X_{ub}$ in the LP. Set $t_{\max} = 2$.

\While{there exists a violated inequality of form~\eqref{eq3k},}{
    Solve the LP to obtain a safe lower bound $\bar{f}_{lb}$ and an optimal solution $X_{lb}$. \\
    \If{$\bar{f}_{lb}>f_{lb}$,}
       {Update $f_{lb} = \bar{f}_{lb}$}
    Round $X_{lb}$ and get a partition matrix $\bar{X}_{ub}$ with cost $\bar{f}_{ub}$.\\
    \If{$\bar{f}_{ub} < f_{ub}$,}
        {Update $X_{ub} = \bar{X}_{ub}$ and
        $f_{ub} = \bar{f}_{ub}$}
    Update $r_{g} = (f_{ub}-f_{lb})/f_{ub}$\\
    \If{$r_{g} \leq \epsilon_{\rm opt}$,}{Terminate}
    Remove from the LP, inequalities~\eqref{eq3k} that are not satisfied tightly at $X_{lb}$.\\
    Run Algorithm~\ref{alg:sep} to obtain 
    a number of inequalities of the form~\eqref{eq3k}
    with $t \leq t_{\max}$ that are violated at $X_{lb}$. 
    Add at most $p_{\max}$ of the most violated inequalities to the LP.\\
    
    \If{the number of violated inequalities returned by Algorithm~\sep is small,}
    {Update $t_{max} = \min\{K, t_{max} + 1\}$}
}
\caption{The cutting-plane algorithm for solving Problem~\eqref{lp:pK}}
\label{alg:Cutting}
\end{algorithm}

An overview of the cutting-plane algorithm is given in Algorithm~\ref{alg:Cutting}.  The algorithm iteratively adds violated inequalities of the form~\eqref{eq3k} to the current LP until no more violated inequalities can be found or the optimality gap is below a given tolerance. Moreover, the sparsity of the inequalities added to the LP is controlled by the parameter $t_{\max}$, which is initially set to $t_{\max} = 2$, and is increased by one only if the number of violated inequalities~\eqref{eq3k} with $t \leq t_{\max}$ found by Algorithm~\ref{alg:sep} is below a certain threshold.

In this paper, we extend Algorithm~\ref{alg:Cutting} to solve two important variants of data clustering: fair K-means clustering and spectral clustering. As we mentioned in Section~\ref{sec:intro}, and we will further detail in the next section, an LP relaxation for fair K-means clustering is obtained by adding a linear number of linear inequalities to Problem~\eqref{lp:pK}. Moreover, an LP relaxation for spectral clustering is obtained by replacing the objective function coefficients $d_{ij}$ by $L_{ij}$ in Problem~\eqref{lp:pK}. Therefore, components $(ii)$ and~$(iv)$ of Algorithm~\ref{alg:Cutting}, \ie the construction of safe lower bounds and the separation algorithm remain unchanged. However, as we describe in the next sections, for each clustering problem, we design customized algorithms to carry out components $(i)$ and~$(iii)$; \ie the initialization and rounding steps.

\section{Fair K-means clustering}
\label{sec:fair_clustering}

In this section, we consider the fair K-means clustering problem, where fairness is defined using the notion of disparate impact~\cite{feldman15}. As we described in Section~\ref{sec:intro}, this notion of group fairness is often quantified in terms of clustering balance~\cite{chier17}, and is enforced in various ways, such as proportional representation defined by inequalities~\eqref{faircond}~\cite{bera19} or $\tau$-ratio fairness defined by inequalities~\eqref{taufair}~\cite{gupta23}.

\subsection{Formulating fairness constraints as linear inequalities}

The next two propositions indicate that the inequalities~\eqref{faircond} and~\eqref{taufair} can be expressed as linear inequalities in terms of the partition matrices defined by~\eqref{pm}. In the following, given a set of data points $\X=\{x^i\}_{i=1}^n$, denote by $\{\Gamma_k\}_{k=1}^K$ a clustering of these points with the associated partition matrix $X$ defined by~\eqref{pm}. Suppose that each point has a protected attribute $g \in \G$, where $\G$ denotes the set of protected groups.  For each $g \in \G$, denote by $\X_g \subseteq \X$ the set of points with protected attribute $g$.
In addition, for each $g \in \G$, denote by $\eta^g \in \{0,1\}^n$ the indicator vector of group $g$; \ie $\eta^g_{i} = 1$ if $x^i \in \X_g$ and $\eta^g_{i} = 0$, otherwise. 

\begin{proposition}\label{fairineq1}
Let $\{\Gamma_k\}_{k=1}^K$ denote a clustering of $n$ points with the associated partition matrix $X$ defined by~\eqref{pm}. 
Then $\{\Gamma_k\}_{k=1}^K$ satisfies inequalities~\eqref{faircond}
if and only if $X$ satisfies the following inequalities:
\begin{equation}\label{pf1}
   \beta_g \leq  \sum_{i = 1}^{n} X_{ij}\eta^g_{i} \leq \alpha_g, \quad \forall g\in \G,\; \forall j\in [n]. 
\end{equation}
\end{proposition}
\begin{proof}
  Fix a group $g \in \G$. For each $j \in [n]$, denote by $k(j) \in [K]$
  the unique index such that $j \in \Gamma_{k(j)}$. From the definition of a partition matrix $X$ and the indicator vector $\eta^g$ it follows that
  \begin{equation}\label{usethis}
  \sum_{i = 1}^{n} X_{ij}\eta^g_{i} = \sum_{i \in \Gamma_{k(j)}}{\frac{1}{|\Gamma_{k(j)}|} \eta^g_{i}}= \frac{1}{|\Gamma_{k(j)}|} \sum_{i \in \Gamma_{k(j)}}{ \eta^g_{i}}=\frac{|\Gamma_{k(j)} \cap \X_g|}{|\Gamma_{k(j)}|},
  \end{equation}
  where the last equality follows since the expression $\sum_{i \in \Gamma_{k(j)}}{ \eta^g_{i}}$ counts the number of points in $\Gamma_{k(j)}$ that belong to group $g$. Therefore, for any $g \in \G$, we have 
  $$
  \beta_g \leq \frac{|\Gamma_{k} \cap \X_g|}{|\Gamma_{k}|} \leq \alpha_g, \quad \forall k \in [K],
  $$
  if and only if
$$
 \beta_g \leq  \sum_{i = 1}^{n} X_{ij}\eta^g_{i} \leq \alpha_g, \quad \forall j\in [n],
$$
and this completes the proof.
\end{proof}

\begin{proposition}\label{fairineq2}
Let $\{\Gamma_k\}_{k=1}^K$ denote a clustering of $n$ points with the associated partition matrix $X$ defined by~\eqref{pm}. 
Then $\{\Gamma_k\}_{k=1}^K$ satisfies inequalities~\eqref{taufair}
if and only if $X$ satisfies the following inequalities:
\begin{equation}\label{pf2}
   \sum_{i = 1}^{n} X_{ij}\eta^g_{i} \geq \tau_g |\X_g| X_{jj}, \quad \forall g\in \G,\; \forall j\in [n]. 
\end{equation}
\end{proposition}
\begin{proof}
Fix a group $g \in \G$. For each $j \in [n]$, denote by $k(j) \in [K]$
  the unique index such that $j \in \Gamma_{k(j)}$. From identity~\eqref{usethis} and the definition of a partition matrix $X$, it follows that
  $$
  \big(\sum_{i = 1}^{n} X_{ij}\eta^g_{i}\big)\cdot \frac{1}{X_{jj}|\X_g|} = \frac{|\Gamma_{k(j)} \cap \X_g|}{|\X_g|}.
  $$
Therefore, for any $g \in \G$, we have 
  $$
  \frac{|\Gamma_{k} \cap \X_g|}{|\X_g|} \geq \tau_g, \quad \forall k \in [K],
  $$
  if and only if
$$
  \sum_{i = 1}^{n} X_{ij}\eta^g_{i} \geq \tau_g |\X_g| X_{jj}, \quad \forall j\in [n],
$$
and this completes the proof.  
\end{proof}

\subsection{The cutting-plane algorithm for fair K-means clustering}
\label{subsec:fairAlg}

In this section, we present a cutting-plane algorithm for solving the fair K-means clustering problem. The overall structure of this algorithm is similar to that of Algorithm~\ref{alg:Cutting} for solving the K-means clustering problem.
Specifically, the algorithm
consists of the four main components described in Section~\ref{sec:cutting_plane}. We next elaborate on each of these components.

By Propositions~\ref{fairineq1} and~\ref{fairineq2}, an LP relaxation for fair K-means clustering is obtained by adding either inequalities~\eqref{pf1} or inequalities~\eqref{pf2} 
to Problem~\eqref{lp:pK}. Notice that the number of such fairness inequalities grows linearly with the number of data points, implying that the bottleneck in solving the resulting LPs is the family of inequalities~\eqref{eq3k}. Hence, the safe lower-bounding and separation components of Algorithm~\ref{alg:Cutting} remain unchanged. It remains to specify the initialization and rounding components, both of which aim to produce a \emph{fair partition matrix}; \ie a partition matrix satisfying inequalities~\eqref{pf1} or inequalities~\eqref{pf2}.

To obtain fair partition matrices, we utilize a variant of fair Lloyd's algorithm, which was first proposed in~\cite{lawgun24}. Recall that Lloyd's algorithm alternates between two steps until convergence is reached: assigning each point to the closest centroid and updating each centroid as the mean of its assigned points. The algorithm of~\cite{lawgun24} replaces the assignment step with an integer programming problem that minimizes the total within-cluster squared distance subject to some fairness constraints. In our adaptation, the assignment step uses either constraints~\eqref{faircond} or constraints~\eqref{taufair}, to enforce our notion of fairness. For completeness, next we present this heuristic, which we will refer to as the {\it fair Lloyd's algorithm}.

Let $z_{ik}$ be a binary variable equal to $1$ if point $i$ is assigned to cluster $k$ and $0$ otherwise, and for each $k \in [K]$, let $c_{k}\in\R^m$ denote the centroid of cluster $k$. 
Given fixed centroids $c_1, \cdots, c_K$, the \emph{fair assignment problem} under constraints~\eqref{faircond}, is given by:
\begin{align}
\label{fairAssignI}
\tag{fairAssignI}
\min \quad & \sum_{i\in [n]}\sum_{k\in [K]}z_{ik}\|x^{i}-c_{k}\|^{2}_{2}\nonumber\\
\text{s.t.} \quad & \sum_{k\in [K]}z_{ik}=1, \quad \forall i\in [n],\nonumber\\
& \beta_g \sum_{i \in [n]} z_{ik}\leq \sum_{i \in [n]}z_{ik}\eta^g_{i} \leq \alpha_g  \sum_{i \in [n]} z_{ik}, \quad \forall g\in \G,\; \forall k\in [K],\label{fairAssignCons}\\
& z_{ik}\in \{0,1\}, \quad \forall i\in [n],\; k\in [K]\nonumber
\end{align}
Similarly, the fair assignment problem 
under constraints~\eqref{taufair} is given by:
\begin{align}
\label{fairAssignII}
\tag{fairAssignII}
\min \quad & \sum_{i\in [n]}\sum_{k\in [K]}z_{ik}\|x^{i}-c_{k}\|^{2}_{2}\nonumber\\
\text{s.t.} \quad & \sum_{k\in [K]}z_{ik}=1, \quad \forall i\in [n],\nonumber\\
& \sum_{i \in [n] } z_{ik}\eta^g_{i} \geq \lceil\tau_g |\X_g|\rceil, \quad \forall g\in \G,\; k\in [K].\label{fairAssignConsGupta}\\
& z_{ik}\in \{0,1\}, \quad \forall i\in [n],\; k\in [K]\nonumber
\end{align}
The fair Lloyd's algorithm iterates between solving Problems~\eqref{fairAssignI} or~\eqref{fairAssignII} fixing the current centroids and updating the centroids as the means of the assigned points, as summarized in Algorithm~\ref{alg:fairLloyd}.
\begin{algorithm}[htb]
    \caption{fair Lloyd's Algorithm}
    \label{alg:fairLloyd}
    \KwIn{Dataset $\{x^i\}_{i=1}^n$, number of clusters $K$, group information $\eta^g_{i}$, $i \in [n]$, $g \in \G$, and initial centroids $\{c_k\}_{k=1}^K$}
    \KwOut{Cluster assignments $z_{ik}$, $i \in [n]$, $k \in [K]$, and final centroids $\{c_k\}_{k=1}^K$}
    \Repeat{$f_1 = f_2$}{
        Given fixed centroids $\{c_k\}_{k=1}^K$, solve the fair assignment problem (\ie Problem~\eqref{fairAssignI} or Problem~\eqref{fairAssignII}) to obtain assignments $z^*$\;
Compute the cost of $z^*$ denoted by $f_1$\;
        \For{$k \in [K]$}{
            Update centroid $c_k = {\sum_{i=1}^n z_{ik} x^i} / {\sum_{i=1}^n z_{ik}}$\;
        }
        Compute the cost of $z^*$ with updated centroids denoted by $f_2$
    }
\end{algorithm}

We should remark that Algorithm~\ref{alg:fairLloyd} terminates after finitely many iterations. The argument is identical to the proof for Lloyd's algorithm, by noting that there are finitely many partitions of points into $K$ clusters and since the objective function decreases in every iteration, no partition can be visited twice.
The fair Lloyd's algorithm serves two purposes within our algorithm: it is used in the initialization step to generate an initial fair partition matrix, and it serves as the rounding heuristic to convert the LP solution to a fair partition matrix (see Algorithm~\ref{alg:rounding}).

\begin{algorithm}[htb]
    \caption{The rounding scheme to obtain a fair partition matrix}
    \label{alg:rounding}
    \KwIn{Data points: $\{x^i\}_{i=1}^n$, number of clusters: $K$ and, the LP solution $X_{lb}$}
    \KwOut{Fair partition matrix $X_{ub}$}
    Compute the eigenvectors $v_k$, $k \in [K]$, of $X_{lb}$ corresponding to its $K$ largest eigenvalues.
    Compute the initial cluster centers:
    $$
    c_{k} = v_{k}^{T}W,\quad \forall k \in [K],
    $$
    where $W$ denotes the matrix whose $i$th column is $x^i$\;
    Apply Algorithm~\ref{alg:fairLloyd} using $\{c_{k}\}_{k=1}^{K}$ as the initial centroids to obtain a fair partition matrix $X_{ub}$.
\end{algorithm}

Let us now discuss the computational cost of the fair Lloyd's algorithm. It is well-known that the traditional Lloyd's algorithm is extremely efficient and often produces high-quality solutions. However, in fair Lloyd's algorithm the trivial assignment step of Lloyd's algorithm is replaced by solving an integer programming problem, which can be computationally expensive in general. Next, we consider the cost of solving this integer program.
In the following, by \emph{the LP relaxation} of the fair assignment problem, we imply the LP obtained from Problem~\eqref{fairAssignI} or Problem~\eqref{fairAssignII}  by replacing 
$z_{ik} \in \{0,1\}$ by $z_{ik} \in [0,1]$ for all $i \in [n]$ and $k \in [K]$.
The next proposition indicates that the LP relaxation of Problem~\eqref{fairAssignII} is tight.

\begin{proposition}\label{prop:taufairLP}
The feasible region of the LP relaxation of Problem~\eqref{fairAssignII} is integral.
\end{proposition}

\begin{proof}
Consider the LP relaxation of Problem~\eqref{fairAssignII}.
Since $0 \le z_{ik}\le 1$, for all $i \in [n]$, and $k \in [K]$,  the inequalities
$$
\sum_{i\in[n]}z_{ik}\eta_i^g\le |\mathcal X_g|,
\qquad \forall g \in\G, k \in [K],
$$
are redundant for this LP. Adding them to the LP relaxation, the feasible region of this problem can be written as
$$
Q=\{z:\ c\le Az\le d,\;0\leq z \leq 1\},
$$
where $A$ is the matrix whose rows are indexed by
$[n]\cup(\G\times[K])$ and whose columns are indexed by
$[n]\times[K]$. Let $g_i\in\G$ denote the protected group containing point $i$. In $A$,
the column corresponding to $z_{ik}$ has a
coefficient of $1$ in row $i$, a coefficient of $1$ in row
$(g_i,k)$, and $0$ elsewhere. Moreover,
$c_i=d_i=1$ for all $i\in[n]$, and
$c_{(g,k)}=\lceil\tau_g|\mathcal X_g|\rceil$, $d_{(g,k)}=|\mathcal X_g|$, for all $(g,k)\in\G\times[K]$. Hence $c$ and $d$ are integral vectors. 

Now consider the graph $H$ with node set $[n]\cup(\G\times[K])$ and the edge set $\{(i,(g_i,k)):i\in[n],\,k\in[K]\}$.
The graph $H$ is bipartite with bipartition
$[n]$ and $\G\times[K]$, and by construction, $A$ is its node-edge incidence
matrix. Therefore, by
Theorem~4.18 in~\cite{ConCorZam12b}, the matrix $A$ is totally unimodular.
Finally, since $A$ is totally unimodular and
$c,d$ are integral vectors,
Theorem~4.5 in~\cite{ConCorZam12b} implies that
$Q$ is an integral polyhedron.
\end{proof}

In~\cite{bera19}, the authors assert the NP-hardness of Problem~\eqref{fairAssignI} in a general metric space without providing a proof. In the following, for completeness, we provide a proof of NP-hardness for this problem in our special setting where the costs are squared Euclidean distances.

\begin{proposition}\label{prop:alfafairNPhard}
   Problem~\eqref{fairAssignI} is NP-hard even with $|\G|=2$.
\end{proposition}

\begin{proof}
We reduce from the NP-complete perfect three dimensional matching (3DM) problem, which can be stated as follows. Given three disjoint sets $X,Y,Z$ of equal cardinality $q$, and a set $T \subseteq X \times Y \times Z$, decide whether there exists $M \subseteq T$ such that every element of $X\cup Y\cup Z$ appears in exactly one triple of $M$. 
Given $X,Y,Z$ and $T$, we construct a fair assignment instance as follows. Define $U := X \cup Y \cup Z$.
Our ambient space is $\mathbb R^{|U|}$, with standard basis vectors denoted by $\{e_u: u\in U\}$. For every $u\in U$, create a point $x^u:=e_u$. For every triple $t =(x,y,z)\in T$, create a center $c_t:=e_x+e_y+e_z$. Thus $K=|T|$.
This construction is clearly polynomial in the size of the 3DM instance.
Notice that empty clusters are allowed in
Problem~\eqref{fairAssignI}, since if $\sum_i z_{ik}=0$, then the fairness constraints reduce to $0\le 0\le 0$.
Now, compute the squared distances; two cases arise:
\begin{itemize}
\item If $u\in t$, then
$\|x^u-c_t\|_2^2=\|e_u-(e_x+e_y+e_z)\|_2^2=2$.

\item If $u\notin t$, then
$\|x^u-c_t\|_2^2=\|e_u-(e_x+e_y+e_z)\|_2^2=4$.
\end{itemize}
We define two protected groups $R$ and $B$ as:
$$
R := X \cup Y, \qquad B:=Z,
$$
and we set the fairness parameters as:
\begin{equation}\label{paramset}
\alpha_R = \beta_R = \frac23, \quad \alpha_B = \beta_B = \frac13.
\end{equation}
Consider the decision version of the fair assignment problem asking whether there is a fair assignment whose cost is at most $6q$.

If the 3DM instance has a perfect matching $M$, assign each point $u\in U$ to the unique triple $t\in M$ containing it. By the above distance computations, every assigned point has cost $2$, so the total cost is
$2|U|=2\cdot 3q=6q$.
Moreover, each nonempty cluster contains exactly two points from $R$ and one point from $B$, so the fairness constraints are satisfied.

Conversely, suppose that there is a fair assignment of cost at most $6q$. Since
there are $3q$ points and each assignment cost is at least $2$, every point must be assigned exactly with cost $2$. Therefore, each point $u$ is assigned
to a center $c_t$ such that $u\in t$. Now consider any nonempty cluster assigned to a center $c_t$, where
$t=(x,y,z)$. Since all assigned points must belong to $t$, the cluster is a
nonempty subset of $\{x,y,z\}$. The only such subset satisfying fairness constraints with parameters~\eqref{paramset}  is the entire set
$\{x,y,z\}$.
Hence, every nonempty cluster is exactly one triple from $T$.
Since each point is assigned to exactly one cluster, the nonempty clusters are pairwise disjoint. Moreover, every point belongs to some cluster, so the union of the nonempty clusters is $X \cup Y \cup Z$. As shown above, each nonempty cluster is exactly one triple from $T$. Therefore, the nonempty clusters form pairwise disjoint triples covering $X \cup Y \cup Z$, that is, they define a perfect 3DM.

Thus, we proved that the constructed fair assignment instance has a feasible assignment of cost
at most $6q$ if and only if the original perfect 3DM
instance is feasible. Hence, the decision version of Problem~\eqref{fairAssignI} is NP-hard. Since a polynomial-time algorithm for the optimization problem would also solve the decision problem by comparing the optimal objective value with the threshold $6q$, it follows that Problem~\eqref{fairAssignI} is NP-hard as well.
\end{proof}

Perhaps surprisingly, as we will show in the next section, the computational cost of the fair Lloyd's algorithm is not visibly affected by the choice of fairness constraints. Namely, for instances with $n \leq 3000$, thanks to {\tt Gurobi}'s highly efficient cutting-plane technology, the computational cost of solving the integer program~\eqref{fairAssignI} is comparable to that of the LP relaxation of~\eqref{fairAssignII} (see Table~\ref{table:fairLloydTiming}).

\subsection{Numerical experiments}

To evaluate the performance of the proposed algorithm, we conduct experiments on real-world data sets\footnote{The source code as well as the data sets are available at
\href{https://github.com/Yakun1125/cutLPK/}{\tt https://github.com/Yakun1125/cutLPK/}.}, available from the UCI Machine Learning Repository~\cite{dua2017uci}. The data sets are summarized in Table~\ref{table:data stats}, where for each data set, we list the number of points (size) and the data balance as defined by~\eqref{balance}. The {\tt Titanic} data set provides two natural candidates for the protected attribute: gender (male and female) and passenger class (1st, 2nd, and 3rd class). We therefore create two variants: {\tt Titanic 2}, which treats gender as the protected attribute, and {\tt Titanic 3}, which treats passenger class as the protected attribute. For all data sets, the protected attribute is excluded from distance computations. 
To assess the behavior of the proposed algorithm on larger data sets, we include two additional data sets obtained by sampling $2000$ and $3000$ points from the {\tt Adult} data set, respectively.

\begin{table}[htp]
\footnotesize
\caption{Summary of the real-world data sets for fair K-means clustering.}
\centering
\begin{tabular}{|c|cc|}
    \hline
    Data set & size & balance \\ \hline
    Heart Disease Hungarian (HH) & 294 & 0.38 \\
    Heart Disease Cleveland (HC) & 297 & 0.48 \\
    Students Math (SM) & 395 & 0.90 \\
    WDBC & 569 & 0.59 \\
    Students Portuguese (SP) & 649 & 0.69 \\
    Titanic 2 & 721 & 0.57 \\
    Titanic 3  & 721 & 0.49 \\
    Credit & 1000 & 0.45 \\
    Adult sample 1 (AS1) & 2000 & 0.50 \\
    Adult sample 2 (AS2) & 3000 & 0.50 \\\hline
\end{tabular}
\label{table:data stats}
\end{table}

To impose group fairness, we consider two variants, namely, inequalities~\eqref{pf1} and inequalities~\eqref{pf2}.
For inequalities~\eqref{pf1}, we set
$$
\alpha_g=\frac{|\X_g|}{\rho n}, \quad \beta_g= \frac{\rho |\X_g|}{n}, \quad \forall g \in \G
,$$
where $\rho \in (0,1]$. 
We refer to inequalities~\eqref{pf1} with the above choice of parameters as \emph{$\alpha$-fair constraints}. In this formulation, setting $\rho=1$ indicates that the proportion of each protected group in every cluster matches its proportion in the full data set. For inequalities~\eqref{pf2}, we set 
$$
\tau_g = \frac{\rho}{K}, \quad \forall g \in \G,
$$
where again $\rho \in (0,1]$.
We refer to inequalities~\eqref{pf2} with the above choice of parameters as \emph{$\tau$-fair constraints}. In this formulation, setting $\rho = 1$ indicates that each cluster contains a $\frac{1}{K}$ fraction of every protected group.
For each data set, we set $K \in \{2, 3, 4, 5\}$ and  $\rho \in \{0.99, 0.9, 0.8, 0.7\}$. We do not consider $\rho = 1$, because it often leads to infeasible problems.
For each combination of $K$ and $\rho$, we solve the fair K-means clustering problem once with $\alpha$-fair constraints and once with $\tau$-fair constraints using our proposed cutting-plane algorithm.

Our cutting-plane algorithm is initialized with $p_{\rm init} = 10^6$ inequalities and terminates if at least one of the following conditions is satisfied:
\begin{itemize}
\item The relative optimality gap
$\delta_{g} = \frac{f_{\rm ub} - f_{\rm lb}} {f_{\rm ub}}$ 
is smaller than the optimality tolerance $\epsilon_{\rm opt}=10^{-4}$. 
\item The run time exceeds the time limit  $T= 10{,}800$ seconds.
\item No violated inequalities~\eqref{eq3k} are found by Algorithm~\ref{alg:sep} within $300$ seconds. 
\item  $\delta_g$ does not decrease after four consecutive iterations with $t_{\max} = K$.
\end{itemize}

All experiments are conducted on {\tt Google Colab} using an Intel(R) Xeon(R) CPU @ 2.20GHz with 8 cores and 50 GB of RAM; the GPU is an NVIDIA G4 with 95.6 GB of RAM. We use {\tt cuPDLPx}~\cite{lu2025cupdlpx}, a GPU-based first-order solver, for solving the LP relaxations and {\tt Gurobi}~\cite{gurobi} for solving Problems~\eqref{fairAssignI} and~\eqref{fairAssignII}.

\paragraph{Results overview.}
In total, $286$ instances were tested; of these, $253$ instances achieved a relative optimality gap below $1\%$, and all but three instances achieved a relative optimality gap below $2\%$. This highlights the strength of the proposed LP relaxation for fair K-means clustering and is in agreement with the computational results of~\cite{derKhaWan24} regarding the strength of the LP relaxation for the (unfair) K-means clustering.
All instances with $n \leq 1000$ terminate within the three-hour time limit. Among the instances sampled from the {\tt Adult} data set  with $n \in \{2000, 3000\}$, $18$ instances exceed the time limit, though the algorithm produces solutions with small optimality gaps in most cases.
From Figure~\ref{fig:performance_profile}, it can be seen  that nearly $40\%$ of the instances reach an optimality gap below $0.1\%$ within $2000$ seconds.
In fact, more than $72\%$ of the instances achieve an optimality gap below $1\%$ within $500$ seconds.
This is important because in almost all clustering applications, an optimality gap of $1\%$ is sufficient. It is worth noting that in $131$ out of $286$ instances, our algorithm finds better solutions than those found by the fair Lloyd's algorithm in the initialization step (see Tables~\ref{table:ratio_fairness_objbal}).

\begin{figure}[htp]
    \centering
    \includegraphics[width=0.5\textwidth]{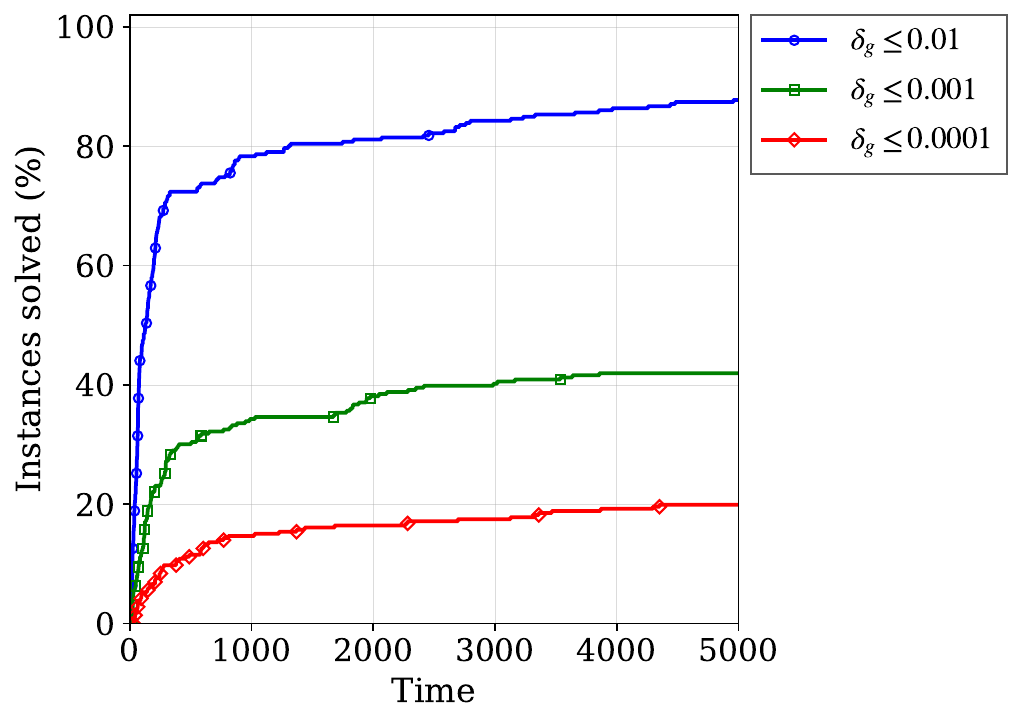}
    \caption{Performance profile for 286 instances for fair K-means clustering. The curves show the percentage of instances whose relative optimality gap $\delta_g$ is at most $10^{-2}$, $10^{-3}$, and $10^{-4}$ within a given time limit.}
    \label{fig:performance_profile}
\end{figure}

\paragraph{Price of fairness.}
The detailed results are given in  Tables~\ref{table:ratio_fairness_time}--\ref{table:ratio_fairness_objbal}.
For comparison, for each data set, we solve the clustering problem without fairness constraints; \ie we use Algorithm~\ref{alg:Cutting} to solve the (unfair) K-means clustering problem.
If a given choice of $\rho$ yields a clustering balance nearly identical to that of the (unfair) K-means clustering, we omit those cases.  The results show that, in the case of data sets for which the K-means solution is highly unbalanced, including the fairness constraints leads to a significant improvement in the clustering balance. For example, for {\tt WDBC} with $K=3,4,5$, the balance of the K-means solution is zero, whereas setting $\rho=0.99$, both fair formulations yield a clustering balance around $0.58$--$0.59$. Similarly, for {\tt HH} with $K=4,5$, the balance of the K-means solution is zero, while the fair formulations with $\rho=0.99$ yield clustering balances around $0.36$--$0.38$. Perhaps surprisingly, this improvement does not necessarily come with a large increase in the objective value: for data sets {\tt SM}, {\tt SP}, {\tt AS1}, {\tt AS2}, and {\tt Credit}, the objective values of the fair formulations remain close to those of the unfair formulation. However, experiments with data sets {\tt WDBC}, {\tt Titanic 2}, and {\tt Titanic 3} indicate that the price of fairness can be substantial, especially when $K$ is large. The computational cost of imposing fairness is also evident from Table~\ref{table:ratio_fairness_time}: the algorithm for solving the (unfair) K-means clustering terminates in at most a few hundred seconds for all instances with $n \leq 1000$, and, on these instances, solving the unfair formulation is on average about $2.5$ times faster than solving the fair K-means clustering with $\alpha$-fair or $\tau$-fair constraints.
Moreover, upon termination, the unfair algorithm often returns a solution with a smaller relative optimality gap.
These observations reflect the additional complexity of adding fairness constraints to Problem~\eqref{lp:pK}.

\begingroup
\footnotesize
\setlength{\tabcolsep}{3pt}
\begin{longtable}{|lcc c|cc|cc|cc|}
\caption{Time (seconds) and relative optimality gap $\delta_g$ for the fair K-means clustering with $\alpha$-fair constraints ($\alpha$-fair), fair K-means clustering with $\tau$-fair constraints ($\tau$-fair), and K-means clustering (unfair ) for data sets considered.}\label{table:ratio_fairness_time}\\
\hline
Data set & $n$ & $K$ & $\rho$ & \multicolumn{2}{c|}{$\alpha$-fair} & \multicolumn{2}{c|}{$\tau$-fair} & \multicolumn{2}{c|}{unfair} \\
\cline{5-10}
 &  &  &  & Time & $\delta_g$ & Time & $\delta_g$ & Time & $\delta_g$ \\
\hline
\endfirsthead
\caption[]{(continued)}\\
\hline
Data set & $n$ & $K$ & $\rho$ & \multicolumn{2}{c|}{$\alpha$-fair} & \multicolumn{2}{c|}{$\tau$-fair} & \multicolumn{2}{c|}{unfair} \\
\cline{5-10}
 &  &  &  & Time & $\delta_g$ & Time & $\delta_g$ & Time & $\delta_g$ \\
\hline
\endhead
\hline
\endfoot
\hline
\endlastfoot
\multirow{14}{*}[-2pt]{HH} & \multirow{14}{*}[-2pt]{294} & 2 & 0.99 & 123.8 & $5.4\times 10^{-4}$ & 70.1 & $9.0\times 10^{-5}$ & 70.0 & $6.4\times 10^{-5}$ \\
 & & & 0.9 & 45.0 & $3.2\times 10^{-4}$ & 44.8 & $1.0\times 10^{-4}$ &  &     \\
 & & & 0.8 & 29.7 & $8.2\times 10^{-5}$ & 288.7 & $1.0\times 10^{-4}$ &  &     \\
 & & 3 & 0.99 & 123.7 & $6.4\times 10^{-3}$ & 36.8 & $1.6\times 10^{-5}$ & 43.9 & $3.3\times 10^{-5}$ \\
 & & & 0.9 & 75.5 & $1.5\times 10^{-3}$ & 100.5 & $1.2\times 10^{-4}$ &  &     \\
 & & & 0.8 & 77.7 & $2.5\times 10^{-5}$ & 75.2 & $3.4\times 10^{-4}$ &  &     \\
 & & 4 & 0.99 & 252.5 & $3.4\times 10^{-2}$ & 79.2 & $6.9\times 10^{-3}$ & 45.0 & $3.0\times 10^{-5}$ \\
 & & & 0.9 & 113.3 & $5.3\times 10^{-3}$ & 139.3 & $1.8\times 10^{-3}$ &  &     \\
 & & & 0.8 & 121.5 & $1.0\times 10^{-2}$ & 132.3 & $6.1\times 10^{-3}$ &  &     \\
 & & & 0.7 & 152.8 & $1.4\times 10^{-2}$ & 118.9 & $7.8\times 10^{-3}$ &  &     \\
 & & 5 & 0.99 & 232.2 & $5.5\times 10^{-2}$ & 78.3 & $9.5\times 10^{-3}$ & 28.2 & $4.5\times 10^{-5}$ \\
 & & & 0.9 & 169.4 & $7.8\times 10^{-3}$ & 120.4 & $7.2\times 10^{-3}$ &  &     \\
 & & & 0.8 & 204.9 & $1.1\times 10^{-2}$ & 125.4 & $1.1\times 10^{-2}$ &  &     \\
 & & & 0.7 & 182.2 & $1.5\times 10^{-2}$ & 117.2 & $1.0\times 10^{-2}$ &  &     \\
\hline
\multirow{16}{*}[-2pt]{HC} & \multirow{16}{*}[-2pt]{297} & 2 & 0.99 & 92.8 & $1.0\times 10^{-4}$ & 58.1 & $8.7\times 10^{-5}$ & 50.7 & $2.7\times 10^{-5}$ \\
 & & & 0.9 & 42.8 & $9.1\times 10^{-5}$ & 65.0 & $8.0\times 10^{-6}$ &  &     \\
 & & & 0.8 & 93.2 & $2.0\times 10^{-4}$ & 63.3 & $4.2\times 10^{-5}$ &  &     \\
 & & & 0.7 & 30.5 & $1.4\times 10^{-4}$ & 69.7 & $2.9\times 10^{-5}$ &  &     \\
 & & 3 & 0.99 & 86.0 & $1.1\times 10^{-2}$ & 67.2 & $9.1\times 10^{-4}$ & 22.1 & $3.1\times 10^{-5}$ \\
 & & & 0.9 & 58.8 & $1.6\times 10^{-3}$ & 32.4 & $3.8\times 10^{-5}$ &  &     \\
 & & & 0.8 & 47.4 & $1.4\times 10^{-4}$ & 95.9 & $12.3\times 10^{-4}$ &  &     \\
 & & & 0.7 & 34.7 & $1.2\times 10^{-4}$ & 74.5 & $9.6\times 10^{-4}$ &  &     \\
 & & 4 & 0.99 & 100.8 & $4.8\times 10^{-3}$ & 155.1 & $3.7\times 10^{-3}$ & 47.6 & $5.2\times 10^{-5}$ \\
 & & & 0.9 & 92.6 & $3.5\times 10^{-3}$ & 103.5 & $1.9\times 10^{-3}$ &  &     \\
 & & & 0.8 & 85.8 & $2.6\times 10^{-3}$ & 135.9 & $3.1\times 10^{-4}$ &  &     \\
 & & & 0.7 & 102.3 & $2.2\times 10^{-3}$ & 126.8 & $1.1\times 10^{-3}$ &  &     \\
 & & 5 & 0.99 & 99.5 & $9.3\times 10^{-3}$ & 195.4 & $6.6\times 10^{-3}$ & 28.8 & $4.9\times 10^{-5}$ \\
 & & & 0.9 & 103.4 & $3.6\times 10^{-3}$ & 189.2 & $5.3\times 10^{-3}$ &  &     \\
 & & & 0.8 & 106.8 & $4.6\times 10^{-3}$ & 286.4 & $4.6\times 10^{-3}$ &  &     \\
 & & & 0.7 & 147.0 & $4.9\times 10^{-3}$ & 134.9 & $2.1\times 10^{-3}$ &  &     \\
\hline
\multirow{13}{*}[-2pt]{SM} & \multirow{13}{*}[-2pt]{395} & 2 & 0.99 & 44.7 & $9.6\times 10^{-5}$ & 100.4 & $7.8\times 10^{-5}$ & 39.8 & $6.1\times 10^{-5}$ \\
 & & 3 & 0.99 & 134.3 & $7.0\times 10^{-4}$ & 152.5 & $2.7\times 10^{-3}$ & 72.7 & $6.1\times 10^{-5}$ \\
 & & & 0.9 & 148.2 & $8.9\times 10^{-4}$ & 163.2 & $3.2\times 10^{-3}$ &  &     \\
 & & & 0.8 & 139.4 & $1.5\times 10^{-4}$ & 154.0 & $2.5\times 10^{-3}$ &  &     \\
 & & & 0.7 & 96.7 & $5.2\times 10^{-5}$ & 161.1 & $2.1\times 10^{-3}$ &  &     \\
 & & 4 & 0.99 & 141.6 & $8.8\times 10^{-4}$ & 174.2 & $1.9\times 10^{-3}$ & 100.0 & $7.8\times 10^{-5}$ \\
 & & & 0.9 & 165.0 & $2.0\times 10^{-3}$ & 179.6 & $2.2\times 10^{-3}$ &  &     \\
 & & & 0.8 & 150.8 & $6.3\times 10^{-4}$ & 192.3 & $1.9\times 10^{-3}$ &  &     \\
 & & & 0.7 & 115.2 & $1.3\times 10^{-4}$ & 174.9 & $8.3\times 10^{-4}$ &  &     \\
 & & 5 & 0.99 & 166.3 & $2.8\times 10^{-3}$ & 182.4 & $4.4\times 10^{-3}$ & 188.2 & $1.5\times 10^{-4}$ \\
 & & & 0.9 & 201.5 & $2.6\times 10^{-3}$ & 168.0 & $3.5\times 10^{-3}$ &  &     \\
 & & & 0.8 & 170.1 & $2.0\times 10^{-3}$ & 171.3 & $2.2\times 10^{-3}$ &  &     \\
 & & & 0.7 & 171.1 & $1.5\times 10^{-3}$ & 169.6 & $1.9\times 10^{-3}$ &  &     \\
\hline
\multirow{6}{*}[-2pt]{WDBC} & \multirow{6}{*}[-2pt]{569} & 2 & 0.99 & 843.3 & $1.3\times 10^{-3}$ & 373.4 & $1.6\times 10^{-3}$ & 270.8 & $7.8\times 10^{-5}$ \\
 & & & 0.9 & 310.8 & $7.4\times 10^{-3}$ & 346.9 & $5.6\times 10^{-3}$ &  &      \\
 & & & 0.8 & 179.1 & $8.5\times 10^{-3}$ & 257.1 & $6.2\times 10^{-3}$ &  &      \\
 & & & 0.7 & 86.7 & $6.9\times 10^{-3}$ & 294.1 & $5.2\times 10^{-3}$ &  &      \\
 & & 3 & 0.99 & 163.8 & $2.4\times 10^{-3}$ & 219.6 & $2.0\times 10^{-3}$ & 167.7 & $5.9\times 10^{-6}$ \\
 & & & 0.9 & 233.4 & $8.2\times 10^{-3}$ & 309.9 & $1.3\times 10^{-2}$ &  &      \\
\multirow{10}{*}[-2pt]{WDBC} & \multirow{10}{*}[-2pt]{569} & 3 & 0.8 & 243.4 & $7.9\times 10^{-3}$ & 313.3 & $1.1\times 10^{-2}$ &  &      \\
 & & & 0.7 & 270.6 & $5.1\times 10^{-3}$ & 287.2 & $4.0\times 10^{-3}$ &  &      \\
 & & 4 & 0.99 & 189.0 & $1.4\times 10^{-3}$ & 655.2 & $1.4\times 10^{-3}$ & 135.3 & $8.6\times 10^{-5}$ \\
 & & & 0.9 & 207.0 & $7.9\times 10^{-3}$ & 353.7 & $1.9\times 10^{-2}$ &  &      \\
 & & & 0.8 & 267.1 & $1.0\times 10^{-2}$ & 350.3 & $1.6\times 10^{-2}$ &  &      \\
 & & & 0.7 & 281.3 & $7.3\times 10^{-3}$ & 332.9 & $8.5\times 10^{-3}$ &  &      \\
 & & 5 & 0.99 & 240.7 & $4.2\times 10^{-3}$ & 373.2 & $6.6\times 10^{-3}$ & 154.2 & $2.8\times 10^{-5}$ \\
 & & & 0.9 & 279.3 & $1.1\times 10^{-2}$ & 408.3 & $2.0\times 10^{-2}$ &  &      \\
 & & & 0.8 & 334.3 & $1.4\times 10^{-2}$ & 432.2 & $2.0\times 10^{-2}$ &  &      \\
 & & & 0.7 & 404.3 & $1.2\times 10^{-2}$ & 388.1 & $1.3\times 10^{-2}$ &  &      \\
\hline
\multirow{14}{*}[-2pt]{SP} & \multirow{14}{*}[-2pt]{649} & 2 & 0.99 & 141.2 & $8.9\times 10^{-5}$ & 148.0 & $7.1\times 10^{-5}$ & 98.1 & $2.6\times 10^{-5}$ \\
 & & & 0.9 & 88.4 & $4.7\times 10^{-5}$ & 164.9 & $5.1\times 10^{-5}$ &  &      \\
 & & 3 & 0.99 & 295.4 & $5.4\times 10^{-4}$ & 253.4 & $8.1\times 10^{-5}$ & 215.3 & $5.0\times 10^{-5}$ \\
 & & & 0.9 & 338.4 & $3.7\times 10^{-4}$ & 215.4 & $7.6\times 10^{-5}$ &  &      \\
 & & & 0.8 & 270.0 & $1.7\times 10^{-4}$ & 207.4 & $1.0\times 10^{-4}$ &  &      \\
 & & & 0.7 & 280.6 & $9.8\times 10^{-5}$ & 206.3 & $8.3\times 10^{-5}$ &  &      \\
 & & 4 & 0.99 & 409.8 & $2.0\times 10^{-3}$ & 423.3 & $2.8\times 10^{-3}$ & 284.8 & $9.3\times 10^{-5}$ \\
 & & & 0.9 & 403.6 & $1.4\times 10^{-3}$ & 466.0 & $2.5\times 10^{-3}$ &  &      \\
 & & & 0.8 & 382.4 & $6.6\times 10^{-4}$ & 481.0 & $2.6\times 10^{-3}$ &  &      \\
 & & & 0.7 & 418.1 & $2.4\times 10^{-4}$ & 439.0 & $2.3\times 10^{-3}$ &  &      \\
 & & 5 & 0.99 & 428.9 & $4.6\times 10^{-3}$ & 458.0 & $7.2\times 10^{-3}$ & 428.3 & $2.8\times 10^{-3}$ \\
 & & & 0.9 & 441.8 & $5.0\times 10^{-3}$ & 467.7 & $6.0\times 10^{-3}$ &  &      \\
 & & & 0.8 & 418.5 & $5.4\times 10^{-3}$ & 436.1 & $5.4\times 10^{-3}$ &  &      \\
 & & & 0.7 & 410.3 & $5.0\times 10^{-3}$ & 466.9 & $5.2\times 10^{-3}$ &  &      \\
\hline
\multirow{15}{*}[-2pt]{Titanic 2} & \multirow{15}{*}[-2pt]{721} & 2 & 0.99 & 589.1 & $7.4\times 10^{-4}$ & 477.7 & $2.9\times 10^{-4}$ & 317.5 & $6.8\times 10^{-5}$ \\
 & & & 0.9 & 454.9 & $1.4\times 10^{-3}$ & 406.9 & $1.2\times 10^{-4}$ &  &      \\
 & & & 0.8 & 553.5 & $3.4\times 10^{-4}$ & 450.3 & $1.1\times 10^{-4}$ &  &      \\
 & & & 0.7 & 579.6 & $2.1\times 10^{-4}$ & 1225.0 & $9.1\times 10^{-5}$ &  &      \\
 & & 3 & 0.99 & 1574.3 & $4.6\times 10^{-3}$ & 413.4 & $5.7\times 10^{-4}$ & 272.8 & $3.1\times 10^{-5}$ \\
 & & & 0.9 & 1124.7 & $1.3\times 10^{-2}$ & 697.1 & $1.6\times 10^{-3}$ &  &      \\
 & & & 0.8 & 1097.9 & $3.5\times 10^{-3}$ & 631.6 & $2.4\times 10^{-3}$ &  &      \\
 & & 4 & 0.99 & 1060.2 & $3.2\times 10^{-3}$ & 709.4 & $7.3\times 10^{-4}$ & 233.8 & $8.4\times 10^{-5}$ \\
 & & & 0.9 & 949.5 & $4.8\times 10^{-3}$ & 618.9 & $9.3\times 10^{-3}$ &  &      \\
 &  &  & 0.8 & 373.6 & $1.2\times 10^{-4}$ & 577.0 & $1.5\times 10^{-2}$ &  &      \\
 & & & 0.7 & 614.8 & $3.1\times 10^{-4}$ & 598.1 & $6.7\times 10^{-3}$ &  &      \\
 & & 5 & 0.99 & 911.8 & $5.1\times 10^{-3}$ & 703.1 & $5.3\times 10^{-3}$ & 184.8 & $1.6\times 10^{-7}$ \\
 & & & 0.9 & 797.0 & $7.4\times 10^{-3}$ & 654.0 & $3.3\times 10^{-3}$ &  &      \\
 & & & 0.8 & 411.3 & $2.5\times 10^{-3}$ & 475.6 & $1.0\times 10^{-3}$ &  &      \\
 & & & 0.7 & 770.6 & $3.3\times 10^{-4}$ & 768.9 & $1.2\times 10^{-3}$ &  &      \\
\hline
\multirow{13}{*}[-2pt]{Titanic 3} & \multirow{13}{*}[-2pt]{721} & 2 & 0.99 & 1127.4 & $6.2\times 10^{-4}$ & 1214.6 & $1.3\times 10^{-4}$ & 132.5 & $8.7\times 10^{-5}$ \\
 & & & 0.9 & 244.9 & $4.3\times 10^{-4}$ & 665.3 & $1.7\times 10^{-4}$ &  &      \\
 & & & 0.8 & 328.6 & $1.0\times 10^{-4}$ & 334.9 & $6.5\times 10^{-4}$ &  &      \\
 & & & 0.7 & 181.4 & $4.3\times 10^{-5}$ & 381.0 & $7.9\times 10^{-5}$ &  &      \\
 & & 3 & 0.99 & 1089.7 & $6.6\times 10^{-3}$ & 874.7 & $5.8\times 10^{-3}$ & 199.3 & $4.2\times 10^{-5}$ \\
 & & & 0.9 & 860.4 & $6.4\times 10^{-3}$ & 441.1 & $5.1\times 10^{-3}$ &  &      \\
 & & & 0.8 & 1124.3 & $9.3\times 10^{-3}$ & 519.8 & $4.7\times 10^{-3}$ &  &      \\
 & & & 0.7 & 786.7 & $1.4\times 10^{-2}$ & 688.9 & $1.0\times 10^{-3}$ &  &      \\
 & & 4 & 0.99 & 671.6 & $9.5\times 10^{-3}$ & 654.0 & $1.3\times 10^{-2}$ & 157.6 & $3.3\times 10^{-5}$ \\
 & & & 0.9 & 600.3 & $6.2\times 10^{-3}$ & 431.7 & $2.6\times 10^{-2}$ &  &      \\
 & & & 0.8 & 801.7 & $1.3\times 10^{-2}$ & 456.8 & $1.4\times 10^{-2}$ &  &      \\
 & & & 0.7 & 720.9 & $1.2\times 10^{-2}$ & 576.5 & $9.0\times 10^{-3}$ &  &      \\
 & & 5 & 0.99 & 970.6 & $1.4\times 10^{-2}$ & 1270.5 & $1.1\times 10^{-2}$ & 209.6 & $6.7\times 10^{-5}$ \\
 \multirow{3}{*}[-2pt]{Titanic 3} & \multirow{3}{*}[-2pt]{721} & 5 & 0.9 & 642.5 & $8.4\times 10^{-3}$ & 508.8 & $1.8\times 10^{-2}$ &  &      \\
 & & & 0.8 & 507.0 & $8.5\times 10^{-3}$ & 470.6 & $1.4\times 10^{-2}$ &  &      \\
 & & & 0.7 & 617.2 & $1.5\times 10^{-2}$ & 483.3 & $6.9\times 10^{-3}$ &  &      \\
\hline
\multirow{16}{*}[-2pt]{Credit} & \multirow{16}{*}[-2pt]{1000} & 2 & 0.99 & 616.2 & $8.0\times 10^{-5}$ & 1023.9 & $1.4\times 10^{-4}$ & 226.7 & $8.3\times 10^{-5}$ \\
 & & & 0.9 & 416.1 & $6.5\times 10^{-5}$ & 403.6 & $8.0\times 10^{-5}$ &  &      \\
 & & & 0.8 & 578.0 & $8.6\times 10^{-5}$ & 379.9 & $9.8\times 10^{-5}$ &  &      \\
 & & & 0.7 & 244.1 & $9.7\times 10^{-5}$ & 458.8 & $3.8\times 10^{-5}$ &  &      \\
 & & 3 & 0.99 & 1012.0 & $1.6\times 10^{-4}$ & 1280.4 & $1.5\times 10^{-4}$ & 575.8 & $3.9\times 10^{-5}$ \\
 & & & 0.9 & 944.1 & $6.6\times 10^{-4}$ & 1397.8 & $7.3\times 10^{-4}$ &  &      \\
 & & & 0.8 & 645.8 & $9.1\times 10^{-5}$ & 1342.8 & $1.4\times 10^{-3}$ &  &      \\
 & & & 0.7 & 581.7 & $4.8\times 10^{-5}$ & 1186.4 & $1.1\times 10^{-3}$ &  &      \\
 & & 4 & 0.99 & 958.0 & $4.3\times 10^{-4}$ & 1833.0 & $1.1\times 10^{-3}$ & 757.7 & $5.7\times 10^{-5}$ \\
 & & & 0.9 & 989.4 & $2.1\times 10^{-4}$ & 1512.3 & $2.0\times 10^{-3}$ &  &      \\
 & & & 0.8 & 953.7 & $1.9\times 10^{-4}$ & 1411.8 & $1.5\times 10^{-3}$ &  &      \\
 & & & 0.7 & 811.5 & $8.5\times 10^{-5}$ & 1288.7 & $4.5\times 10^{-4}$ &  &      \\
 & & 5 & 0.99 & 1398.3 & $1.1\times 10^{-4}$ & 1023.0 & $8.9\times 10^{-5}$ & 848.8 & $4.9\times 10^{-5}$ \\
 & & & 0.9 & 946.5 & $1.3\times 10^{-4}$ & 808.6 & $8.8\times 10^{-5}$ &  &      \\
 & & & 0.8 & 488.5 & $8.1\times 10^{-5}$ & 733.2 & $8.2\times 10^{-5}$ &  &      \\
 & & & 0.7 & 722.8 & $5.6\times 10^{-5}$ & 586.1 & $7.4\times 10^{-5}$ &  &      \\
 \hline
 \multirow{13}{*}[-2pt]{AS1} & \multirow{13}{*}[-2pt]{2000} & 2 & 0.99 & 1682.3 & $8.2\times 10^{-5}$ & $>$10800 & $3.9\times 10^{-4}$ & 1566.3 & $1.1\times 10^{-5}$ \\
 & & & 0.9 & 1369.8 & $2.0\times 10^{-5}$ & 2265.8 & $9.7\times 10^{-5}$ &  &      \\
 & & & 0.8 & 1435.4 & $2.9\times 10^{-5}$ & 2282.4 & $9.0\times 10^{-6}$ &  &      \\
 & & 3 & 0.99 & 7447.7 & $9.6\times 10^{-5}$ & 8817.3 & $9.5\times 10^{-5}$ & 2755.4 & $9.5\times 10^{-6}$ \\
 & & & 0.9 & 3359.9 & $4.0\times 10^{-5}$ & 3330.8 & $9.6\times 10^{-5}$ &  &      \\
 & & 4 & 0.99 & 3738.6 & $2.4\times 10^{-4}$ & 9089.6 & $2.2\times 10^{-4}$ & 2884.6 & $9.9\times 10^{-5}$ \\
 & & & 0.9 & 3912.6 & $2.2\times 10^{-4}$ & 5257.9 & $1.5\times 10^{-4}$ &  &      \\
 & & & 0.8 & 4251.0 & $1.1\times 10^{-4}$ & 3868.9 & $9.9\times 10^{-5}$ &  &      \\
 & & & 0.7 & 2695.2 & $7.6\times 10^{-5}$ & 3466.8 & $5.6\times 10^{-5}$ &  &      \\
 & & 5 & 0.99 & 4483.4 & $2.0\times 10^{-4}$ & 7852.1 & $6.1\times 10^{-3}$ & 3363.4 & $9.7\times 10^{-5}$ \\
 & & & 0.9 & 3972.2 & $1.4\times 10^{-4}$ & 6060.1 & $9.0\times 10^{-3}$ &  &      \\
 & & & 0.8 & 3128.5 & $6.4\times 10^{-5}$ & 5538.1 & $4.0\times 10^{-3}$ &  &      \\
 & & & 0.7 & 4351.8 & $1.1\times 10^{-6}$ & 5447.6 & $2.7\times 10^{-3}$ &  &      \\
\hline
\multirow{10}{*}[-2pt]{AS2} & \multirow{10}{*}[-2pt]{3000} & 2 & 0.99 & $>$10800 & $1.7\times 10^{-4}$ & $>$10800 & $8.8\times 10^{-4}$ & 5574.2 & $1.9\times 10^{-6}$ \\
 &  & & 0.9 & 4297.4 & $2.3\times 10^{-5}$ & 7613.0 & $9.3\times 10^{-5}$ &  &      \\
 & & 3 & 0.99 & $>$10800 & $5.4\times 10^{-4}$ & $>$10800 & $6.5\times 10^{-4}$ & 9595.2 & $2.3\times 10^{-5}$ \\
 & & & 0.9 & 8204.5 & $2.4\times 10^{-5}$ & $>$10800 & $6.5\times 10^{-4}$ &  &      \\
 & & 4 & 0.99 & $>$10800 & $2.0\times 10^{-4}$ & $>$10800 & $1.1\times 10^{-3}$ & 8334.5 & $2.4\times 10^{-5}$ \\
 & & & 0.9 & $>$10800 & $1.1\times 10^{-4}$ & $>$10800 & $5.7\times 10^{-4}$ &  &      \\
 & & & 0.8 & $>$10800 & $1.5\times 10^{-4}$ & $>$10800 & $2.3\times 10^{-4}$ &  &      \\
 & & 5 & 0.99 & $>$10800 & $3.9\times 10^{-3}$ & $>$10800 & $9.4\times 10^{-3}$ & $>$10800 & $3.0\times 10^{-4}$ \\
 & & & 0.9 & $>$10800 & $5.4\times 10^{-4}$ & $>$10800 & $1.3\times 10^{-2}$ &  &      \\
 & & & 0.8 & $>$10800 & $6.3\times 10^{-4}$ & $>$10800 & $7.9\times 10^{-3}$ &  &      \\
\end{longtable}
\endgroup

\begingroup
\footnotesize
\setlength{\tabcolsep}{3pt}
\begin{longtable}{|lcc c|cc|cc|cc|}
\caption{Objective value (cost) and clustering balance (balance) for the fair K-means clustering with $\alpha$-fair constraints ($\alpha$-fair), fair K-means clustering with $\tau$-fair constraints ($\tau$-fair), and K-means clustering (unfair ) for data sets considered.}\label{table:ratio_fairness_objbal}\\
\hline
Data set & $n$ & $K$ & $\rho$ & \multicolumn{2}{c|}{$\alpha$-fair} & \multicolumn{2}{c|}{$\tau$-fair} & \multicolumn{2}{c|}{unfair} \\
\cline{5-10}
 &  &  &  & cost & balance & cost & balance & cost & balance \\
\hline
\endfirsthead
\caption[]{(continued)}\\
\hline
Data set & $n$ & $K$ & $\rho$ & \multicolumn{2}{c|}{$\alpha$-fair} & \multicolumn{2}{c|}{$\tau$-fair} & \multicolumn{2}{c|}{unfair} \\
\cline{5-10}
 &  &  &  & cost & balance & cost & balance & cost & balance \\
\hline
\endhead
\hline
\endfoot
\hline
\endlastfoot
HH & 294 & 2 & 0.99 & 3072.1 & 0.38 & 3184.1 & 0.38 & 3046.0 & 0.25 \\
 \multirow{13}{*}[-2pt]{HH} & \multirow{13}{*}[-2pt]{294} & 2& 0.9 & 3052.7 & 0.33 & 3145.2 & 0.38 &  &  \\
 & & & 0.8 & 3047.3 & 0.29 & 3103.7 & 0.38 &  &  \\
 & & 3 & 0.99 & 2836.6 & 0.38 & 2877.5$^*$ & 0.38 & 2759.0 & 0.26 \\
 & & & 0.9 & 2797.1 & 0.33 & 2860.9$^*$ & 0.33 &  &  \\
 & & & 0.8 & 2788.8 & 0.29 & 2853.8$^*$ & 0.29 &  &  \\
 & & 4 & 0.99 & 2664.4 & 0.38 & 2763.7$^*$ & 0.37 & 2497.4 & 0.00 \\
 & & & 0.9 & 2560.0 & 0.33 & 2725.0$^*$ & 0.28 &  &  \\
 & & & 0.8 & 2556.6 & 0.29 & 2719.0$^*$ & 0.25 &  &  \\
 & & & 0.7 & 2554.9 & 0.24 & 2707.7 & 0.25 &  &  \\
 & & 5 & 0.99 & 2528.7 & 0.38 & 2653.2 & 0.36 & 2238.9 & 0.00 \\
 & & & 0.9 & 2367.8$^*$ & 0.33 & 2620.1$^*$ & 0.28 &  &  \\
 & & & 0.8 & 2355.0 & 0.28 & 2601.9 & 0.23 &  &  \\
 & & & 0.7 & 2342.6 & 0.24 & 2574.0 & 0.29 &  &  \\
\hline
\multirow{16}{*}[-2pt]{HC} & \multirow{16}{*}[-2pt]{297} & 2 & 0.99 & 4512.5 & 0.47 & 4535.4 & 0.48 & 4456.5 & 0.24 \\
 & & & 0.9 & 4490.4 & 0.41 & 4508.0 & 0.40 &  &  \\
 & & & 0.8 & 4473.7 & 0.35 & 4490.1 & 0.38 &  &  \\
 & & & 0.7 & 4463.1 & 0.29 & 4476.1 & 0.34 &  &  \\
 & & 3 & 0.99 & 4198.0 & 0.47 & 4221.7$^*$ & 0.48 & 4122.5 & 0.25 \\
 & & & 0.9 & 4166.8 & 0.41 & 4155.8 & 0.37 &  &  \\
 & & & 0.8 & 4141.2 & 0.35 & 4144.4 & 0.30 &  &  \\
 & & & 0.7 & 4128.9$^*$ & 0.29 & 4135.3 & 0.28 &  &  \\
 & & 4 & 0.99 & 3933.1 & 0.47 & 4010.4 & 0.47 & 3819.6 & 0.23 \\
 & & & 0.9 & 3888.7 & 0.41 & 3927.9$^*$ & 0.35 &  &  \\
 & & & 0.8 & 3859.9 & 0.35 & 3881.8$^*$ & 0.28 &  &  \\
 & & & 0.7 & 3840.1$^*$ & 0.29 & 3860.9 & 0.24 &  &  \\
 & & 5 & 0.99 & 3702.8 & 0.47 & 3822.9$^*$ & 0.46 & 3528.2 & 0.06 \\
 & & & 0.9 & 3632.6 & 0.41 & 3750.6$^*$ & 0.34 &  &  \\
 & & & 0.8 & 3598.4 & 0.35 & 3699.7 & 0.27 &  &  \\
 & & & 0.7 & 3576.5$^*$ & 0.29 & 3658.2 & 0.24 &  &  \\
\hline
\multirow{13}{*}[-2pt]{SM} & \multirow{13}{*}[-2pt]{395} & 2 & 0.99 & 8017.9 & 0.89 & 8034.1 & 0.90 & 8017.9 & 0.88 \\
 & & 3 & 0.99 & 7619.4 & 0.88 & 7661.1$^*$ & 0.90 & 7559.2 & 0.46 \\
 & & & 0.9 & 7599.2 & 0.75 & 7638.3 & 0.89 &  &  \\
 & & & 0.8 & 7578.4 & 0.65 & 7617.2$^*$ & 0.72 &  &  \\
 & & & 0.7 & 7565.8$^*$ & 0.55 & 7601.8 & 0.67 &  &  \\
 & & 4 & 0.99 & 7294.1$^*$ & 0.88 & 7333.6$^*$ & 0.88 & 7218.3 & 0.45 \\
 & & & 0.9 & 7274.4$^*$ & 0.75 & 7302.3$^*$ & 0.66 &  &  \\
 & & & 0.8 & 7242.9$^*$ & 0.61 & 7277.0$^*$ & 0.61 &  &  \\
 & & & 0.7 & 7221.5$^*$ & 0.50 & 7252.8$^*$ & 0.58 &  &  \\
 & & 5 & 0.99 & 7060.1$^*$ & 0.89 & 7098.7$^*$ & 0.84 & 6966.8 & 0.30 \\
 & & & 0.9 & 7025.8$^*$ & 0.74 & 7058.9 & 0.71 &  &  \\
 & & & 0.8 & 6999.7$^*$ & 0.62 & 7023.7$^*$ & 0.59 &  &  \\
 & & & 0.7 & 6982.9$^*$ & 0.50 & 7000.2$^*$ & 0.50 &  &  \\
\hline
\multirow{10}{*}[-2pt]{WDBC} & \multirow{10}{*}[-2pt]{569} & 2 & 0.99 & 14884.5$^*$ & 0.59 & 15016.3$^*$ & 0.58 & 11595.5 & 0.08 \\
 & & & 0.9 & 14665.8 & 0.53 & 14680.1 & 0.49 &  &  \\
 & & & 0.8 & 14262.2 & 0.43 & 14225.2 & 0.40 &  &  \\
 & & & 0.7 & 13728.3 & 0.35 & 13741.1 & 0.32 &  &  \\
 & & 3 & 0.99 & 13915.8$^*$ & 0.59 & 14065.6$^*$ & 0.58 & 10061.8 & 0.00 \\
 & & & 0.9 & 13608.0 & 0.51 & 13618.0 & 0.45 &  &  \\
 & & & 0.8 & 13105.9$^*$ & 0.43 & 13049.0$^*$ & 0.35 &  &  \\
 & & & 0.7 & 12490.8 & 0.35 & 12457.5$^*$ & 0.26 &  &  \\
 & & 4 & 0.99 & 13143.7$^*$ & 0.58 & 13517.9$^*$ & 0.59 & 9257.0 & 0.00 \\
 & & & 0.9 & 12819.2 & 0.51 & 12971.5 & 0.42 &  &  \\
 \multirow{6}{*}[-2pt]{WDBC}& \multirow{6}{*}[-2pt]{569}& 4 & 0.8 & 12346.3 & 0.42 & 12323.2$^*$ & 0.30 &  &  \\
 & & & 0.7 & 11740.9 & 0.35 & 11712.1 & 0.23 &  &  \\
 & & 5 & 0.99 & 12717.9 & 0.59 & 13009.8$^*$ & 0.58 & 8553.5 & 0.00 \\
 & & & 0.9 & 12364.0$^*$ & 0.51 & 12498.2$^*$ & 0.40 &  &  \\
 & & & 0.8 & 11866.7 & 0.42 & 11767.5 & 0.27 &  &  \\
 & & & 0.7 & 11246.5$^*$ & 0.35 & 11107.6$^*$ & 0.24 &  &  \\
\hline
\multirow{14}{*}[-2pt]{SP} & \multirow{14}{*}[-2pt]{649} & 2 & 0.99 & 13017.4$^*$ & 0.69 & 13037.9$^*$ & 0.68 & 12977.2 & 0.58 \\
 & & & 0.9 & 12980.5 & 0.60 & 13007.9 & 0.62 &  &  \\
 & & 3 & 0.99 & 12439.6$^*$ & 0.68 & 12442.9$^*$ & 0.68 & 12320.2 & 0.35 \\
 & & & 0.9 & 12392.5$^*$ & 0.59 & 12382.8$^*$ & 0.54 &  &  \\
 & & & 0.8 & 12349.9 & 0.49 & 12349.1$^*$ & 0.45 &  &  \\
 & & & 0.7 & 12325.5$^*$ & 0.40 & 12330.7$^*$ & 0.39 &  &  \\
 & & 4 & 0.99 & 11952.4$^*$ & 0.68 & 12018.2$^*$ & 0.67 & 11801.1 & 0.36 \\
 & & & 0.9 & 11897.0$^*$ & 0.59 & 11954.1$^*$ & 0.49 &  &  \\
 & & & 0.8 & 11850.9$^*$ & 0.49 & 11913.6$^*$ & 0.42 &  &  \\
 & & & 0.7 & 11818.8$^*$ & 0.40 & 11882.5 & 0.42 &  &  \\
 & & 5 & 0.99 & 11598.2$^*$ & 0.68 & 11703.6$^*$ & 0.67 & 11445.6 & 0.29 \\
 & & & 0.9 & 11550.9 & 0.59 & 11614.9 & 0.48 &  &  \\
 & & & 0.8 & 11516.6$^*$ & 0.49 & 11562.9$^*$ & 0.38 &  &  \\
 & & & 0.7 & 11484.8 & 0.40 & 11523.5$^*$ & 0.36 &  &  \\
\hline
\multirow{15}{*}[-2pt]{Titanic 2} & \multirow{15}{*}[-2pt]{721} & 2 & 0.99 & 4733.2 & 0.57 & 4837.2 & 0.57 & 4465.8 & 0.31 \\
 & & & 0.9 & 4668.7 & 0.52 & 4713.6 & 0.47 &  &  \\
 & & & 0.8 & 4569.7 & 0.43 & 4594.9 & 0.38 &  &  \\
 & & & 0.7 & 4478.8 & 0.34 & 4491.9 & 0.31 &  &  \\
 & & 3 & 0.99 & 3999.6 & 0.57 & 4232.1$^*$ & 0.56 & 3687.7 & 0.29 \\
 & & & 0.9 & 3922.3 & 0.53 & 4070.9 & 0.43 &  &  \\
 & & & 0.8 & 3798.6 & 0.43 & 3909.4 & 0.33 &  &  \\
 & & 4 & 0.99 & 3342.7 & 0.56 & 3864.4$^*$ & 0.56 & 2961.4 & 0.22 \\
 & & & 0.9 & 3223.7 & 0.49 & 3701.8 & 0.40 &  &  \\
 & & & 0.8 & 3108.3$^*$ & 0.41 & 3554.7 & 0.29 &  &  \\
 & & & 0.7 & 3042.4 & 0.34 & 3420.6 & 0.25 &  &  \\
 & & 5 & 0.99 & 3059.7$^*$ & 0.56 & 3587.4 & 0.55 & 2652.5 & 0.14 \\
 & & & 0.9 & 2947.6 & 0.49 & 3399.0 & 0.38 &  &  \\
 & & & 0.8 & 2822.7 & 0.41 & 3208.7 & 0.26 &  &  \\
 & & & 0.7 & 2738.7 & 0.34 & 3069.9 & 0.19 &  &  \\
\hline
\multirow{16}{*}[-2pt]{Titanic 3} & \multirow{16}{*}[-2pt]{721} & 2 & 0.99 & 4515.4 & 0.49 & 4589.9 & 0.48 & 4424.2 & 0.35 \\
 & & & 0.9 & 4467.6 & 0.43 & 4509.9 & 0.43 &  &  \\
 & & & 0.8 & 4430.7 & 0.37 & 4454.3 & 0.35 &  &  \\
 & & & 0.7 & 4424.2 & 0.35 & 4424.2 & 0.35 &  &  \\
 & & 3 & 0.99 & 3955.5 & 0.48 & 4138.3 & 0.48 & 3649.9 & 0.08 \\
 & & & 0.9 & 3881.5 & 0.42 & 4019.6 & 0.40 &  &  \\
 & & & 0.8 & 3827.0 & 0.33 & 3938.1 & 0.30 &  &  \\
 & & & 0.7 & 3795.4 & 0.29 & 3884.1 & 0.28 &  &  \\
 & & 4 & 0.99 & 3547.4 & 0.48 & 3862.6 & 0.47 & 3086.9 & 0.08 \\
 & & & 0.9 & 3456.3$^*$ & 0.41 & 3754.8$^*$ & 0.37 &  &  \\
 & & & 0.8 & 3397.9 & 0.33 & 3605.2$^*$ & 0.26 &  &  \\
 & & & 0.7 & 3315.3 & 0.29 & 3522.8$^*$ & 0.23 &  &  \\
 & & 5 & 0.99 & 3298.8 & 0.48 & 3587.8 & 0.48 & 2724.6 & 0.08 \\
 & & & 0.9 & 3162.6$^*$ & 0.40 & 3455.6 & 0.33 &  &  \\
 & & & 0.8 & 3060.6 & 0.33 & 3305.4$^*$ & 0.24 &  &  \\
 & & & 0.7 & 3003.7 & 0.29 & 3202.8 & 0.21 &  &  \\
\hline
Credit& 1000 & 2 & 0.99 & 5926.4$^*$ & 0.44 & 6098.5 & 0.45 & 5891.4 & 0.28 \\
\multirow{15}{*}[-2pt]{Credit} & \multirow{15}{*}[-2pt]{1000} & 2 & 0.9 & 5909.1$^*$ & 0.39 & 6057.7 & 0.45 &  &  \\
 & & & 0.8 & 5896.3 & 0.33 & 6012.8 & 0.45 &  &  \\
 & & & 0.7 & 5891.4 & 0.28 & 5971.8 & 0.45 &  &  \\
 & & 3 & 0.99 & 5339.5$^*$ & 0.44 & 5450.9$^*$ & 0.44 & 5298.3 & 0.30 \\
 & & & 0.9 & 5319.1$^*$ & 0.39 & 5415.4 & 0.40 &  &  \\
 & & & 0.8 & 5300.2 & 0.33 & 5384.9 & 0.40 &  &  \\
 & & & 0.7 & 5298.3$^*$ & 0.30 & 5355.1 & 0.39 &  &  \\
 & & 4 & 0.99 & 4887.0$^*$ & 0.44 & 5001.5 & 0.44 & 4853.3 & 0.30 \\
 & & & 0.9 & 4865.0 & 0.39 & 4951.6$^*$ & 0.39 &  &  \\
 & & & 0.8 & 4855.4$^*$ & 0.33 & 4907.8 & 0.41 &  &  \\
 & & & 0.7 & 4853.3$^*$ & 0.30 & 4877.5$^*$ & 0.39 &  &  \\
 & & 5 & 0.99 & 4555.9$^*$ & 0.44 & 4665.6$^*$ & 0.44 & 4525.4 & 0.30 \\
 & & & 0.9 & 4539.0$^*$ & 0.39 & 4601.9$^*$ & 0.43 &  &  \\
 & & & 0.8 & 4527.8$^*$ & 0.33 & 4560.5$^*$ & 0.40 &  &  \\
 & & & 0.7 & 4525.5$^*$ & 0.31 & 4540.9$^*$ & 0.36 &  &  \\
\hline
\multirow{13}{*}[-2pt]{AS1} & \multirow{13}{*}[-2pt]{2000} & 2 & 0.99 & 144.0$^*$ & 0.49 & 145.1 & 0.50 & 143.7 & 0.42 \\
 & & & 0.9 & 143.7$^*$ & 0.43 & 144.3 & 0.48 &  &  \\
 & & & 0.8 & 143.7 & 0.42 & 143.8$^*$ & 0.45 &  &  \\
 & & 3 & 0.99 & 117.2$^*$ & 0.49 & 118.6$^*$ & 0.50 & 117.1 & 0.46 \\
 & & & 0.9 & 117.1 & 0.46 & 117.3 & 0.48 &  &  \\
 & & 4 & 0.99 & 106.1$^*$ & 0.49 & 106.6$^*$ & 0.48 & 105.5 & 0.30 \\
 & & & 0.9 & 105.8$^*$ & 0.43 & 105.9$^*$ & 0.37 &  &  \\
 & & & 0.8 & 105.6$^*$ & 0.36 & 105.6$^*$ & 0.34 &  &  \\
 & & & 0.7 & 105.5$^*$ & 0.31 & 105.5$^*$ & 0.33 &  &  \\
 & & 5 & 0.99 & 96.0$^*$ & 0.49 & 99.1 & 0.48 & 95.3 & 0.30 \\
 & & & 0.9 & 95.6$^*$ & 0.43 & 97.9 & 0.37 &  &  \\
 & & & 0.8 & 95.4$^*$ & 0.36 & 96.8$^*$ & 0.36 &  &  \\
 & & & 0.7 & 95.3$^*$ & 0.31 & 96.4 & 0.33 &  &  \\
\hline
\multirow{10}{*}[-2pt]{AS2} & \multirow{10}{*}[-2pt]{3000} & 2 & 0.99 & 220.8$^*$ & 0.50 & 222.1$^*$ & 0.50 & 220.4 & 0.43 \\
 & & & 0.9 & 220.4$^*$ & 0.43 & 221.1$^*$ & 0.46 &  &  \\
 & & 3 & 0.99 & 179.1 & 0.50 & 181.6$^*$ & 0.50 & 178.9 & 0.46 \\
 & & & 0.9 & 178.9 & 0.46 & 179.4$^*$ & 0.50 &  &  \\
 & & 4 & 0.99 & 162.7$^*$ & 0.50 & 163.6$^*$ & 0.49 & 161.7 & 0.32 \\
 & & & 0.9 & 162.1 & 0.43 & 162.4$^*$ & 0.39 &  &  \\
 & & & 0.8 & 161.8$^*$ & 0.37 & 161.9$^*$ & 0.35 &  &  \\
 & & 5 & 0.99 & 147.9 & 0.50 & 152.6$^*$ & 0.48 & 146.4 & 0.32 \\
 & & & 0.9 & 146.8 & 0.43 & 151.0$^*$ & 0.40 &  &  \\
 & & & 0.8 & 146.5$^*$ & 0.37 & 149.4 & 0.35 &  &  \\
\end{longtable}
$^*$ Instances for which our algorithm finds a better solution than the initial fair partition matrix found by fair Lloyd's algorithm in the initialization step.\par
\endgroup

We now discuss the computational cost of the fair Lloyd's algorithm outlined in Algorithm~\eqref{alg:fairLloyd}. We report the cost of this greedy algorithm for the three largest data sets considered in this paper, namely, {\tt Credit}  with $n=1000$, {\tt AS1}  with $n=2000$, and {\tt AS2}  with $n=3000$.  For each combination of $(\rho, K)$,  we run the fair Lloyd's algorithm using five random initializations of the centroids. 
We report the mean and standard deviation of the run time and the number of iterations required for convergence. The results are summarized in Table~\ref{table:fairLloydTiming}.
Overall, the fair Lloyd's algorithm is efficient. Interestingly, the choice between $\alpha$-fair and $\tau$-fair constraints does not noticeably affect the overall computational cost of this algorithm. This behavior is somewhat unexpected because in the case of 
$\tau$-fair constraints, by Proposition~\ref{prop:taufairLP},  to solve Problem~\eqref{fairAssignII}, it suffices to solve an LP, while 
in the case of $\alpha$-fair constraints, by Proposition~\ref{prop:alfafairNPhard}, to solve Problem~\eqref{fairAssignI}, we must solve an NP-hard integer programming problem. In our experiments, we observed that {\tt Gurobi}'s cutting-plane procedure is very effective at closing the optimality gap of Problem~\eqref{fairAssignI} quickly.


\begin{table}[htbp]
\footnotesize
\centering
\caption{The fair Lloyd's algorithm run time and iteration count on larger data sets. For each combination of $\rho,K$, we report the mean with the standard deviation in parentheses across five random trials.}
\label{table:fairLloydTiming}
\begin{tabular}{|lcc|cc|cc|}
\hline
\multirow{2}{*}{Data set} & \multirow{2}{*}{$K$} & \multirow{2}{*}{$\rho$} & \multicolumn{2}{c|}{$\alpha$-fair} & \multicolumn{2}{c|}{$\tau$-fair} \\
 & & & Time (s) & Iterations & Time (s) & Iterations \\
\hline
Credit & 2 & 0.99 & $3.1\,(0.4)$ & $6.6\,(2.1)$ & $2.6\,(0.0)$ & $9.4\,(3.4)$ \\
&   & 0.90 & $2.8\,(0.1)$ & $7.2\,(1.3)$ & $2.6\,(0.1)$ & $8.2\,(2.3)$ \\
& 3 & 0.99 & $3.5\,(0.4)$ & $18.0\,(8.5)$ & $2.6\,(0.0)$ & $17.6\,(5.8)$ \\
&   & 0.90 & $3.3\,(0.3)$ & $17.4\,(6.9)$ & $2.7\,(0.1)$ & $19.8\,(5.2)$ \\
& 4 & 0.99 & $3.5\,(0.5)$ & $12.2\,(6.9)$ & $2.5\,(0.0)$ & $9.0\,(3.4)$ \\
&   & 0.90 & $3.3\,(0.5)$ & $12.4\,(9.7)$ & $2.5\,(0.1)$ & $14.0\,(7.4)$ \\
& 5 & 0.99 & $4.3\,(0.7)$ & $18.4\,(8.4)$ & $3.0\,(0.5)$ & $19.0\,(8.6)$ \\
&   & 0.90 & $6.9\,(1.2)$ & $22.6\,(7.7)$ & $3.8\,(0.6)$ & $20.8\,(3.9)$ \\
\hline
AS1 & 2 & 0.99 & $11.7\,(0.6)$ & $8.2\,(1.9)$ & $11.7\,(1.7)$ & $8.8\,(1.6)$ \\
&   & 0.90 & $14.2\,(0.3)$ & $10.6\,(3.4)$ & $14.7\,(2.0)$ & $11.6\,(1.7)$ \\
& 3 & 0.99 & $14.0\,(0.2)$ & $9.0\,(3.4)$ & $14.6\,(1.6)$ & $15.8\,(3.5)$ \\
&   & 0.90 & $15.8\,(1.2)$ & $8.6\,(3.4)$ & $13.4\,(1.2)$ & $13.4\,(2.7)$ \\
& 4 & 0.99 & $15.1\,(1.3)$ & $20.2\,(7.2)$ & $12.1\,(0.3)$ & $22.6\,(14.3)$ \\
&   & 0.90 & $13.7\,(1.0)$ & $16.6\,(7.2)$ & $13.1\,(1.6)$ & $25.8\,(16.6)$ \\
& 5 & 0.99 & $16.1\,(1.9)$ & $18.2\,(9.1)$ & $15.6\,(4.6)$ & $17.6\,(5.0)$ \\
&   & 0.90 & $16.4\,(0.8)$ & $16.0\,(4.8)$ & $13.7\,(0.5)$ & $20.6\,(7.1)$ \\
\hline
AS2 & 2 & 0.99 & $35.3\,(0.7)$ & $9.0\,(4.5)$ & $35.0\,(0.6)$ & $15.0\,(4.3)$ \\
&   & 0.90 & $35.2\,(0.9)$ & $8.4\,(2.5)$ & $34.5\,(0.4)$ & $14.2\,(4.4)$ \\
& 3 & 0.99 & $35.4\,(1.1)$ & $10.0\,(3.0)$ & $33.6\,(1.1)$ & $14.8\,(5.9)$ \\
&   & 0.90 & $34.9\,(1.5)$ & $10.4\,(2.9)$ & $30.7\,(0.4)$ & $12.6\,(2.6)$ \\
& 4 & 0.99 & $33.8\,(0.9)$ & $14.6\,(4.0)$ & $31.1\,(0.6)$ & $21.0\,(11.1)$ \\
&   & 0.90 & $34.0\,(1.3)$ & $15.0\,(3.2)$ & $31.2\,(0.5)$ & $23.2\,(14.9)$ \\
& 5 & 0.99 & $40.9\,(3.9)$ & $24.6\,(12.0)$ & $31.5\,(1.9)$ & $28.4\,(14.3)$ \\
&   & 0.90 & $37.5\,(6.3)$ & $18.8\,(6.5)$ & $45.5\,(4.5)$ & $25.2\,(14.6)$ \\
\hline
\end{tabular}
\end{table}

Finally, to illustrate the importance of controlling the sparsity of the inequalities added to the LP relaxation by Algorithm~\ref{alg:sep}, selected relative gap trajectories are depicted in Figures~\ref{fig:alpha_fair_math}--\ref{fig:tau_fair_selected} in the Appendix.

\section{Spectral clustering}
\label{sec:spectral_clustering}
In this section, we extend the cutting-plane algorithm of Section~\ref{sec:cutting_plane} to solve the minimum ratio-cut, and hence the spectral clustering problem. As we described in Section~\ref{sec:intro}, spectral clustering is a highly popular two-step heuristic algorithm that aims to solve the  minimum ratio-cut problem defined by~\eqref{minratio}. The outline of this technique is given in Algorithm~\ref{alg:spectral_init}. For further clarity, in the remainder of this paper we refer to this algorithm as the \emph{spectral heuristic}.

\begin{algorithm}[htbp]
    \caption{Spectral heuristic}
    \label{alg:spectral_init}
    \KwIn{Graph Laplacian $L$, number of clusters $K$}
    \KwOut{Cluster assignments for all $n$ points}
    Compute the eigenvectors $u_k$, $k \in [K]$ of $L$ corresponding to its $K$ smallest eigenvalues\;
    Let $U = [u_1, u_2, \cdots, u_K]\in \mathbb{R}^{n \times K}$\;
    Perform K-means clustering on the rows of $U$  using Lloyd's algorithm to obtain the clusters
\end{algorithm}

In~\cite{LinStr20} the authors propose an SDP relaxation for 
Problem~\eqref{minratio} and obtain recovery guarantees for this relaxation under different generative models.
Since the feasible region of Problem~\eqref{minratio} is identical to that of the K-means clustering problem~\eqref{Kmeans2}, replacing $d_{ij}$ by the graph Laplacian entries $L_{ij}$ in Problem~\eqref{lp:pK}, we obtain an LP relaxation for Problem~\eqref{minratio}. 
 In this section, we propose an extension of the cutting-plane algorithm of Section~\ref{sec:cutting_plane} to solve 
Problem~\eqref{minratio}. Subsequently, we demonstrate the effectiveness of our algorithm by performing social network analysis on real-world data sets.

\subsection{The cutting-plane algorithm for spectral clustering}
\label{subsec:spectral_alg}

To extend the cutting-plane algorithm of Section~\ref{sec:cutting_plane} to solve Problem~\eqref{minratio}, we need to specify the initialization and rounding steps. We use the popular spectral heuristic outlined in Algorithm~\ref{alg:spectral_init} to find a good feasible solution and hence an upper bound for Problem~\eqref{minratio}. As we described in Section~\ref{sec:cutting_plane}, this feasible solution is further used for selecting the set of active inequalities~\eqref{eq3k} to construct the first LP in Algorithm~\ref{alg:Cutting}. For the rounding step, we follow a similar strategy to the rounding scheme of K-means clustering described in Algorithm~\ref{alg:rounding}.  The outline of this technique is given in Algorithm~\ref{alg:spectral_round}.

\begin{algorithm}[htb]
    \caption{Rounding scheme for spectral clustering}
    \label{alg:spectral_round}
    \KwIn{Data points $\{x^i\}_{i=1}^n$, number of clusters $K$, LP solution $X_{lb}$}
    \KwOut{Partition matrix $X_{ub}$}
    Compute the eigenvectors $v_i$, $i \in [K]$, of $X_{lb}$ corresponding to its $K$ largest eigenvalues\;
    Construct the embedding matrix $U=[v_1, v_2, \cdots, v_K] \in \mathbb{R}^{n \times K}$\;
    Apply K-means clustering to the rows of $U$ using Lloyd's algorithm to obtain the partition matrix $X_{ub}$.
\end{algorithm}


\subsection{Community detection}
Community detection is one of the most prominent applications of spectral clustering~\cite{Lux07}. Given a network, the goal is to partition its nodes into groups such that nodes within the same community are densely connected, while connections between different communities are sparse. This problem has attracted sustained attention across multiple disciplines, including sociology, biology, computer science, and operations research, and arises in contexts ranging from the study of friendship networks and collaboration graphs to the analysis of protein--protein interaction networks and citation graphs~\cite{Abbe18}.
Community detection can be approached through several graph partitioning formulations. Among the most studied formulations are the minimum bisection,  minimum normalized cut, and minimum ratio-cut problems~\cite{WagWag93}. Theoretical limits of SDPs for community detection have been extensively studied in the literature~\cite{Abbe18}, while the theoretical limits of LPs for community detection were recently investigated in~\cite{dPIdaTim20}.

We now investigate the effectiveness of our cutting-plane algorithm for community detection using real-world data sets. The networks are modeled as weighted or unweighted graphs for which the graph Laplacian can then be computed using equation~\eqref{Laplacian}. 
The graph data are obtained from the following two sources:
\begin{itemize}
    \item University of Michigan network data: \url{https://websites.umich.edu/~mejn/netdata/}
    \item  Stanford Large Network Dataset Collection: 
\url{https://snap.stanford.edu/data/index.html}. 
\end{itemize} 
For the seven smaller networks, namely, {\tt polbooks}, {\tt adjnoun}, {\tt football}, {\tt facebook}, {\tt friendship}, {\tt netscience}, and {\tt deezer ego}, we make use of the entire data set. For the larger networks, namely,  {\tt ca-GrQc}, {\tt ca-HepTh}, and {\tt lastfm\_asia} whose sizes are currently beyond the reach of the proposed algorithm, we generate random samples of size $n \in \{500,1000, 1500\}$. For each data set, we let $K \in \{2,\cdots, 10\}$. The graph of the {\tt friendship} network consists of three connected components. Therefore, the optimal solution for $K=2,3$ is trivial with an objective value of zero. Hence, in this case, we report the results for $K \geq 4$.

Our cutting-plane algorithm is initialized with $p_{\rm init} = 10^7$; all other parameters are set to the same value as in the fair clustering experiments of Section~\ref{sec:fair_clustering}. In these experiments, we selected a higher value for $p_{\rm init}$  because the LP relaxation does not contain additional fairness constraints, and including a larger number of initial cuts often accelerates convergence. As before, all experiments are conducted on {\tt Google Colab} using an Intel(R) Xeon(R) CPU @ 2.20GHz with 8 cores and 50~GB of RAM; the GPU is an NVIDIA G4 with 95.6~GB of RAM. We use {\tt cuPDLPx}~\cite{lu2025cupdlpx} to solve all LP relaxations. 

In total, $142$ instances were tested. Performance profiles are shown in Figure~\ref{fig:spectral_performance_profile}. As can be seen from this figure, 
within the three hour time limit, nearly $60\%$ of instances reach an optimality gap below $0.01\%$, nearly $85\%$ of instances reach an optimality gap below $0.1\%$, and nearly $94\%$ of instances reach an optimality gap below $1\%$. In addition, more than $80\%$ of the instances reach an optimality gap of less than $1.0\%$ in about an hour.

\begin{figure}[htp]
    \centering
    \includegraphics[width=0.5\textwidth]{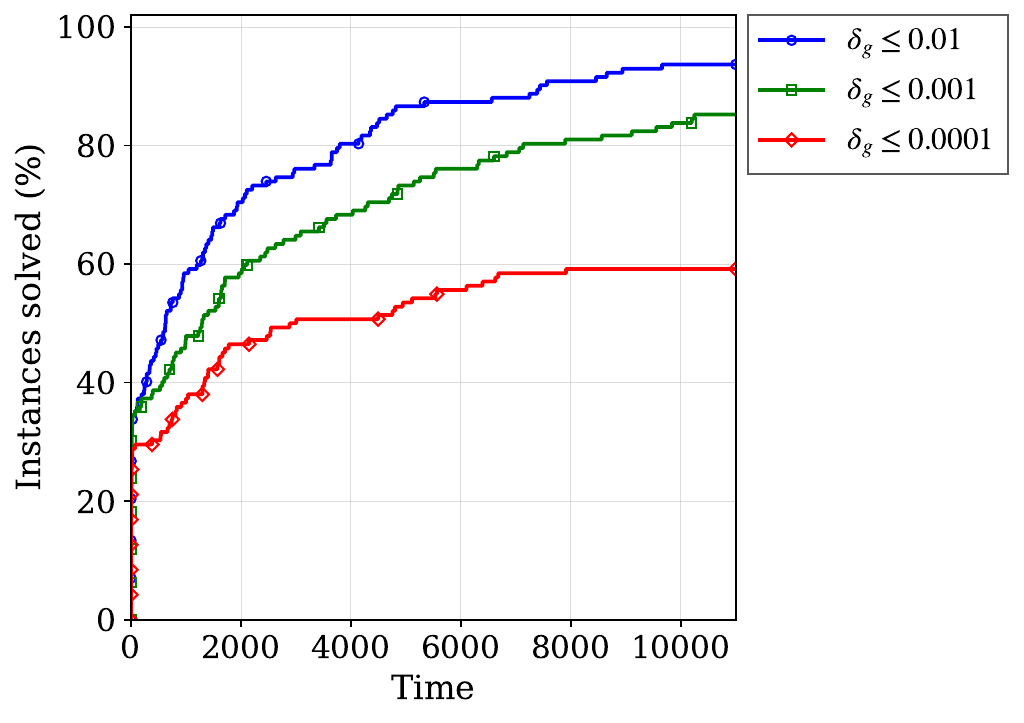}
    \caption{Performance profile for 124 instances for spectral clustering. The curves show the percentage of instances whose relative optimality gap $\delta_g$ is at most $10^{-2}$, $10^{-3}$, and $10^{-4}$ within a given time limit.}
    \label{fig:spectral_performance_profile}
\end{figure}

The detailed results are given in Table~\ref{table:social_network}. For each instance, we report the run time, the maximum separation parameter $t_{\max}$, and the objective value of the solution returned by the cutting-plane algorithm upon termination. For comparison, we also report the objective value of the solution found by the spectral heuristic outlined in Algorithm~\ref{alg:spectral_init}. For the six small networks, the proposed algorithm finds solutions with a relative optimality gap less than $0.01\%$  fairly quickly.
These experiments demonstrate the remarkable strength of the proposed LP relaxation for the minimum-ratio cut problem. 
In fact, the algorithm is efficient because the LP often closes the optimality gap without the need for any dense cutting planes. More specifically,
the separation parameter $t_{\max}$ rarely exceeds $t = 3$ for small networks, and in particular, for the {\tt friendship} network it remains at $t = 2$ throughout.
 Interestingly, for $77\%$ of the instances, the proposed algorithm finds a better solution than the spectral heuristic. More specifically, for $38\%$ of the instances, the proposed algorithm produces a solution whose objective value improves upon that of the spectral heuristic by at least $10\%$.
 This indicates that even though the spectral heuristic is highly efficient, it may produce solutions that are considerably suboptimal. In contrast, the proposed algorithm terminates with a certified optimality gap.

\begingroup
\footnotesize
\setlength{\tabcolsep}{3pt}
\begin{longtable}{|lccccccc|}
\caption{Results on real-world social network data sets. Time is reported in seconds, $t_{\max}$ is the largest separation parameter reached in Algorithm~\ref{alg:Cutting}, and $\delta_g$ is the final relative optimality gap. Best solution is the final solution returned by the proposed algorithm, while the Spectral solution is the solution returned by the spectral heuristic. }
\label{table:social_network}\\
\hline
Dataset & $n$ & $K$ & $t_{\max}$ & Time & $\delta_g$ & Best solution & Spectral solution \\ \hline
\endfirsthead
\caption[]{(continued)}\\
\hline
Dataset & $n$ & $K$ & $t_{\max}$ & Time & $\delta_g$ & Best solution & Spectral solution \\ \hline
\endhead
 \multirow{9}{*}[-2pt]{polbooks} & \multirow{9}{*}[-2pt]{105} & 2 & 2 & 1.8 & $6.3\times 10^{-6}$ & 0.72 & 0.76 \\
 &  & 3 & 3 & 2.6 & $6.2\times 10^{-5}$ & 2.30 & 3.01 \\
 &  & 4 & 4 & 2.0 & $2.1\times 10^{-5}$ & 4.46 & 4.64 \\
 &  & 5 & 4 & 3.9 & $1.3\times 10^{-6}$ & 6.71 & 7.46 \\
 &  & 6 & 3 & 2.5 & $5.0\times 10^{-6}$ & 9.26 & 9.30 \\
 &  & 7 & 7 & 18.8 & $1.0\times 10^{-8}$ & 12.02 & 12.32 \\
 &  & 8 & 4 & 3.1 & $9.5\times 10^{-7}$ & 14.94 & 15.24 \\
 &  & 9 & 3 & 3.4 & $5.4\times 10^{-7}$ & 18.06 & 18.90 \\
 &  & 10 & 3 & 3.3 & $2.4 \times 10^{-7}$ & 21.18 & 23.05 \\ \hline
 \multirow{9}{*}[-2pt]{football} & \multirow{9}{*}[-2pt]{115} & 2 & 2 & 1.2 & $2.7\times 10^{-6}$ & 2.12 & 2.60 \\
 &  & 3 & 3 & 1.4 & $9.7\times 10^{-5}$ & 4.81 & 5.02 \\
 &  & 4 & 3 & 2.1 & $1.8\times 10^{-5}$ & 8.03 & 8.81 \\
 &  & 5 & 3 & 3.2 & $9.5\times 10^{-5}$ & 11.41 & 12.23 \\
 &  & 6 & 3 & 2.6 & $7.0\times 10^{-6}$ & 15.06 & 16.70 \\
 &  & 7 & 3 & 5.3 & $1.8\times 10^{-6}$ & 18.83 & 19.26 \\
 &  & 8 & 3 & 2.5 & $5.0\times 10^{-5}$ & 22.60 & 23.16 \\
 &  & 9 & 3 & 1.3 & $8.8\times 10^{-5}$ & 26.93 & 27.56 \\
 &  & 10 & 3 & 1.5 & $3.0\times 10^{-8}$ & 31.42 & 32.01 \\ \hline
 \multirow{9}{*}[-2pt]{adjnoun} & \multirow{9}{*}[-2pt]{112} & 2 & 2 & 0.4 & $4.7\times 10^{-4}$ & 1.01 & 1.01 \\
 &  & 3 & 3 & 1.3 & $8.3\times 10^{-5}$ & 2.02 & 2.02 \\
 &  & 4 & 4 & 4.2 & $1.6\times 10^{-4}$ & 3.03 & 3.03 \\
 &  & 5 & 5 & 20.8 & $1.7\times 10^{-4}$ & 4.04 & 4.04 \\
 &  & 6 & 6 & 13.9 & $9.9\times 10^{-5}$ & 5.05 & 5.05 \\
 &  & 7 & 7 & 13.5 & $1.3\times 10^{-4}$ & 6.06 & 6.06 \\
 &  & 8 & 2 & 1.8 & $8.8\times 10^{-5}$ & 7.07 & 7.07 \\
 &  & 9 & 2 & 1.6 & $2.6\times 10^{-5}$ & 8.08 & 8.09 \\
 &  & 10 & 2 & 0.8 & $5.2\times 10^{-5}$ & 9.09 & 9.09 \\ \hline
 \multirow{6}{*}[-2pt]{facebook} & \multirow{6}{*}[-2pt]{155} & 2 & 2 & 1.1 & $3.8\times 10^{-4}$ & 1.01 & 1.01 \\
 &  & 3 & 2 & 1.3 & $7.1\times 10^{-5}$ & 2.68 & 2.68 \\
 &  & 4 & 4 & 1.3 & $1.3\times 10^{-4}$ & 5.27 & 5.28 \\
 &  & 5 & 2 & 4.2 & $8.3\times 10^{-5}$ & 8.32 & 8.35 \\
 &  & 6 & 3 & 6.1 & $6.8\times 10^{-5}$ & 11.37 & 11.47 \\
 &  & 7 & 4 & 7.6 & $1.5\times 10^{-6}$ & 14.45 & 14.52 \\\hline
  \multirow{3}{*}[-2pt]{facebook}& \multirow{3}{*}[-2pt]{155} & 8 & 4 & 5.1 & $7.0\times 10^{-7}$ & 17.53 & 17.57 \\
 &  & 9 & 4 & 10.1 & $5.9\times 10^{-6}$ & 21.02 & 21.07 \\
 &  & 10 & 3 & 5.8 & $9.3\times 10^{-6}$ & 25.03 & 25.11 \\ \hline
 \multirow{7}{*}[-2pt]{friendship} & \multirow{7}{*}[-2pt]{133} & 4 & 2 & 1.1 & $2.9\times 10^{-6}$ & 0.26 & 0.26 \\
 &  & 5 & 5 & 3.1 & $2.2\times 10^{-4}$ & 0.64 & 0.76 \\
 &  & 6 & 6 & 11.1 & $4.5\times 10^{-4}$ & 1.16 & 1.40 \\
 &  & 7 & 3 & 1.9 & $4.9\times 10^{-7}$ & 1.74 & 1.93 \\
 &  & 8 & 3 & 1.3 & $1.6\times 10^{-5}$ & 2.33 & 2.86 \\
 &  & 9 & 3 & 5.3 & $9.4\times 10^{-5}$ & 3.13 & 4.28 \\
 &  & 10 & 3 & 3.8 & $7.2\times 10^{-6}$ & 4.17 & 5.37 \\ \hline
 \multirow{9}{*}[-2pt]{deezer\_ego} & \multirow{9}{*}[-2pt]{363} & 2 & 2 & 163.5 & $6.3\times 10^{-4}$ & 1.00 & 1.00 \\
 &  & 3 & 3 & 241.8 & $2.8\times 10^{-4}$ & 2.01 & 2.01 \\
 &  & 4 & 4 & 1036.6 & $9.6\times 10^{-5}$ & 3.01 & 3.01 \\
 &  & 5 & 5 & 1406.0 & $2.1\times 10^{-5}$ & 5.01 & 5.01 \\
 &  & 6 & 4 & 1332.4 & $9.6\times 10^{-5}$ & 7.02 & 7.02 \\
 &  & 7 & 7 & 2289.2 & $2.1\times 10^{-4}$ & 9.03 & 9.03 \\
 &  & 8 & 5 & 538.2 & $3.6\times 10^{-5}$ & 11.03 & 11.03 \\
 &  & 9 & 5 & 809.2 & $2.7\times 10^{-5}$ & 14.04 & 14.04 \\
 &  & 10 & 3 & 664.3 & $2.9\times 10^{-5}$ & 17.05 & 17.05 \\ \hline
 \multirow{9}{*}[-2pt]{netscience} & \multirow{9}{*}[-2pt]{379} & 2 & 2 & 159.0 & $3.7\times 10^{-3}$ & 0.04 & 0.08 \\*
 &  & 3 & 3 & 74.3 & $3.7\times 10^{-5}$ & 0.13 & 0.44 \\*
 &  & 4 & 2 & 12.5 & $1.3\times 10^{-5}$ & 0.23 & 0.25 \\*
 &  & 5 & 2 & 10.0 & $7.2\times 10^{-5}$ & 0.38 & 0.38 \\*
 &  & 6 & 6 & 192.7 & $1.5\times 10^{-4}$ & 0.54 & 0.62 \\*
 &  & 7 & 3 & 26.8 & $1.5\times 10^{-5}$ & 0.74 & 0.91 \\*
 &  & 8 & 2 & 16.4 & $5.3\times 10^{-6}$ & 0.97 & 1.03 \\*
 &  & 9 & 3 & 25.8 & $9.8\times 10^{-5}$ & 1.22 & 1.43 \\*
 &  & 10 & 2 & 21.4 & $2.6\times 10^{-6}$ & 1.49 & 1.88 \\ \hline
 \multirow{9}{*}[-2pt]{ca-GrQc\_500} & \multirow{9}{*}[-2pt]{500} & 2 & 2 & 383.6 & $2.2\times 10^{-5}$ & 0.25 & 0.43 \\
 &  & 3 & 3 & 3526.0 & $3.6\times 10^{-4}$ & 0.57 & 0.79 \\
 &  & 4 & 4 & 1300.8 & $7.2\times 10^{-5}$ & 0.95 & 1.05 \\
 &  & 5 & 3 & 740.6 & $5.5\times 10^{-7}$ & 1.45 & 1.49 \\
 &  & 6 & 3 & 545.7 & $3.1\times 10^{-5}$ & 2.02 & 2.34 \\
 &  & 7 & 3 & 919.0 & $7.4\times 10^{-5}$ & 2.63 & 2.77 \\
 &  & 8 & 3 & 699.9 & $2.1\times 10^{-5}$ & 3.28 & 3.39 \\
 &  & 9 & 3 & 755.3 & $5.2\times 10^{-5}$ & 4.04 & 4.29 \\
 &  & 10 & 3 & 1348.8 & $2.7\times 10^{-5}$ & 4.85 & 5.16 \\ \hline
 \multirow{9}{*}[-2pt]{ca-HepTh\_500} & \multirow{9}{*}[-2pt]{500} & 2 & 2 & 1001.4 & $5.2\times 10^{-5}$ & 0.38 & 0.60 \\
 &  & 3 & 3 & 2360.9 & $1.4\times 10^{-4}$ & 0.86 & 1.12 \\
 &  & 4 & 4 & 2542.0 & $2.3\times 10^{-5}$ & 1.36 & 1.52 \\
 &  & 5 & 4 & 2532.5 & $6.3\times 10^{-5}$ & 1.89 & 2.18 \\
 &  & 6 & 3 & 1609.0 & $6.5\times 10^{-5}$ & 2.50 & 2.61 \\
 &  & 7 & 3 & 1573.0 & $1.3\times 10^{-5}$ & 3.15 & 3.34 \\
 &  & 8 & 8 & 2244.5 & $1.4\times 10^{-3}$ & 3.83 & 4.24 \\
 &  & 9 & 3 & 1400.1 & $3.8\times 10^{-5}$ & 4.51 & 4.74 \\
 &  & 10 & 3 & 832.5 & $9.5\times 10^{-5}$ & 5.22 & 5.68 \\ \hline
 \multirow{7}{*}[-2pt]{lastfm\_asia\_500} & \multirow{7}{*}[-2pt]{500} & 2 & 2 & 3006.3 & $6.7\times 10^{-7}$ & 0.48 & 0.48 \\
 &  & 3 & 3 & 6841.4 & $2.8\times 10^{-4}$ & 0.99 & 0.99 \\
 &  & 4 & 4 & 6626.2 & $3.4\times 10^{-5}$ & 1.49 & 1.49 \\
 &  & 5 & 5 & 686.6 & $1.2\times 10^{-4}$ & 1.99 & 1.99 \\
 &  & 6 & 6 & 6432.2 & $9.0\times 10^{-4}$ & 2.99 & 3.21 \\
 &  & 7 & 7 & $>10800$ & $1.4\times 10^{-3}$ & 4.00 & 4.22 \\
 &  & 8 & 8 & $>10800$ & $4.0\times 10^{-4}$ & 5.00 & 6.45 \\\hline
  \multirow{2}{*}[-2pt]{lastfm\_asia\_500} & \multirow{2}{*}[-2pt]{500} & 9 & 9 & $>10800$ & $1.7\times 10^{-3}$ & 6.00 & 7.93 \\
 &  & 10 & 10 & $>10800$ & $1.7\times 10^{-3}$ & 7.09 & 8.92 \\ \hline
 \multirow{9}{*}[-2pt]{ca-GrQc\_1000} & \multirow{9}{*}[-2pt]{1000} & 2 & 2 & 1302.5 & $3.4\times 10^{-5}$ & 0.17 & 0.17 \\
 &  & 3 & 3 & $>10800$ & $4.7\times 10^{-3}$ & 0.38 & 0.38 \\
 &  & 4 & 4 & 4805.1 & $9.5\times 10^{-5}$ & 0.72 & 0.76 \\
 &  & 5 & 5 & 3891.5 & $1.7\times 10^{-4}$ & 1.08 & 1.33 \\
 &  & 6 & 4 & 1606.1 & $6.9\times 10^{-5}$ & 1.48 & 1.67 \\
 &  & 7 & 7 & 1971.3 & $1.7\times 10^{-4}$ & 1.91 & 2.00 \\
 &  & 8 & 2 & 1709.0 & $4.9\times 10^{-5}$ & 2.41 & 2.92 \\
 &  & 9 & 9 & 4696.2 & $1.5\times 10^{-4}$ & 2.91 & 3.66 \\
 &  & 10 & 3 & 1665.7 & $3.4\times 10^{-5}$ & 3.52 & 4.59 \\ \hline
 \multirow{9}{*}[-2pt]{ca-HepTh\_1000} & \multirow{9}{*}[-2pt]{1000} & 2 & 2 & 1781.5 & $8.1\times 10^{-5}$ & 0.06 & 0.06 \\
 &  & 3 & 3 & $>10800$ & $7.2\times 10^{-3}$ & 0.15 & 0.15 \\
 &  & 4 & 4 & $>10800$ & $1.3\times 10^{-2}$ & 0.32 & 0.53 \\
 &  & 5 & 5 & 4034.6 & $2.6\times 10^{-4}$ & 0.51 & 0.70 \\
 &  & 6 & 6 & 2942.1 & $4.9\times 10^{-4}$ & 0.76 & 0.95 \\
 &  & 7 & 5 & 2149.1 & $9.3\times 10^{-5}$ & 1.01 & 1.16 \\
 &  & 8 & 8 & 1732.2 & $1.2\times 10^{-4}$ & 1.31 & 1.58 \\
 &  & 9 & 9 & 7816.4 & $3.2\times 10^{-4}$ & 1.65 & 1.84 \\
 &  & 10 & 10 & $>10800$ & $1.0\times 10^{-4}$ & 1.98 & 2.15 \\ \hline
 \multirow{9}{*}[-2pt]{lastfm\_asia\_1000} & \multirow{9}{*}[-2pt]{1000} & 2 & 2 & 1444.1 & $2.8\times 10^{-3}$ & 0.11 & 0.11 \\
 &  & 3 & 3 & 9616.0 & $1.0\times 10^{-3}$ & 0.61 & 0.61 \\
 &  & 4 & 4 & 7326.0 & $1.6\times 10^{-4}$ & 1.28 & 1.28 \\
 &  & 5 & 5 & $>10800$ & $2.9\times 10^{-3}$ & 2.28 & 2.69 \\
 &  & 6 & 6 & $>10800$ & $1.6\times 10^{-4}$ & 3.29 & 3.75 \\
 &  & 7 & 7 & 7913.2 & $4.5\times 10^{-5}$ & 4.29 & 4.74 \\
 &  & 8 & 7 & 6395.4 & $4.9\times 10^{-5}$ & 5.29 & 5.76 \\
 &  & 9 & 6 & 6105.2 & $3.6\times 10^{-5}$ & 6.29 & 7.11 \\
 &  & 10 & 7 & 4496.5 & $8.4\times 10^{-5}$ & 7.29 & 7.71 \\ \hline
 \multirow{9}{*}[-2pt]{ca-GrQc\_1500} & \multirow{9}{*}[-2pt]{1500} & 2 & 2 & 5376.3 & $2.7\times 10^{-4}$ & 0.25 & 0.30 \\
 &  & 3 & 3 & $>10800$ & $2.0\times 10^{-1}$ & 0.50 & 0.76 \\
 &  & 4 & 4 & $>10800$ & $3.0\times 10^{-2}$ & 0.75 & 1.02\\
 &  & 5 & 5 &$>10800$ & $1.2\times 10^{-2}$ & 1.00 & 1.32   \\
 &  & 6 & 6 & 9184.4 & $7.6\times 10^{-4}$ & 1.29 & 1.62 \\
 &  & 7 & 7 & 10285.3 & $7.8\times 10^{-4}$ & 1.62 & 1.87 \\
 &  & 8 & 5 & $>10800$ & $3.1\times 10^{-3}$ & 1.96 & 2.25 \\
 &  & 9 & 5 & $>10800$ & $9.0\times 10^{-4}$ & 2.29 & 2.81 \\
 &  & 10 & 10 & $>10800$ & $3.1\times 10^{-4}$ & 2.63 & 3.65 \\ \hline
 \multirow{9}{*}[-2pt]{ca-HepTh\_1500} & \multirow{9}{*}[-2pt]{1500} & 2 & 2 & 2885.0 & $2.8\times 10^{-5}$ & 0.39 & 0.39 \\
 &  & 3 & 3 & $>10800$ & $2.5\times 10^{-1}$ & 0.82 & 0.82 \\
 &  & 4 & 4 & $>10800$ & $6.8\times 10^{-2}$ & 1.25 & 1.25 \\
 &  & 5 & 5 & $>10800$ & $1.0\times 10^{-2}$ & 1.75 & 2.03 \\
 &  & 6 & 5 & 5564.9 & $5.6\times 10^{-5}$ & 2.32 & 2.47 \\
 &  & 7 & 5 & 5113.0 & $8.2\times 10^{-5}$ & 2.92 & 3.23 \\
 &  & 8 & 3 & 5535.2 & $7.0\times 10^{-5}$ & 3.55 & 3.98 \\
 &  & 9 & 4 & 4942.2 & $7.9\times 10^{-5}$ & 4.19 & 4.48 \\
 &  & 10 & 4 & 4759.4 & $9.9\times 10^{-5}$ & 4.88 & 5.33 \\ \hline
 \multirow{6}{*}[-2pt]{lastfm\_asia\_1500} & \multirow{6}{*}[-2pt]{1500} & 2 & 2 & 5563.6 &  $5.5\times 10^{-4}$ & 0.46 & 0.49\\
 &  & 3 & 3 & $>10800$ & $6.1\times 10^{-1}$ & 1.35 & 1.35 \\
 &  & 4 & 4 & $>10800$ & $1.0\times 10^{-2}$ & 1.83 & 2.17 \\
 &  & 5 & 5 & $>10800$ & $5.6\times 10^{-3}$ & 2.56 & 2.95 \\
 &  & 6 & 5 & $>10800$ & $6.1\times 10^{-3}$ & 3.53 & 3.62 \\
 &  & 7 & 6 & 6678.9 & $1.4\times 10^{-5}$ & 4.20 & 4.75 \\\hline
 \multirow{3}{*}[-2pt]{lastfm\_asia\_1500} & \multirow{3}{*}[-2pt]{1500} & 8 & 8 & 9539.17 & $2.0\times 10^{-4}$ & 5.19 & 5.81 \\
 &  & 9 & 8 & $>10800$ & $8.1\times 10^{-4}$ & 6.19 & 6.82 \\
 &  & 10 & 7 & $>10800$ & $8.1\times 10^{-4}$ & 7.19 & 7.81 \\ \hline
\end{longtable}
\endgroup


\begin{footnotesize} 
\bibliographystyle{plainurl}
\bibliography{ref,biblio}
\end{footnotesize}
\newpage
\appendix
\section*{Appendix}
\label{sec: Appendix}

In this appendix, we provide some details on the impact of parameter $t$ on the strength as well as the computational cost of Problem~\eqref{lp:pK}. Recall that Problem~\eqref{lp:pK} contains $\Theta(n^{t+1})$ inequalities of the form~\eqref{eq3k}. Therefore, a careful selection of $t$ is key to the efficiency of the cutting-plane algorithm. 

Figures~\ref{fig:alpha_fair_math}--\ref{fig:tau_fair_selected} depict the effect of increasing $t$ on reducing the relative optimality gap of the proposed cutting-plane algorithm. In most cases, the largest reduction in the optimality gap occurs when separating inequalities~\eqref{eq3k} with $t=2$. Increasing $t$ from $2$ to $3$ leads to an additional visible decrease in the gap for several instances; further increases continue to strengthen the relaxation, but with diminishing returns. We also observe that for some datasets, such as {\tt Student Math}, the improvement thanks to inequalities~\eqref{eq3k} with $t=3$ is more significant for larger $K$ (\ie $K = 4,5$), suggesting that inequalities with larger $t$ become increasingly valuable as the number of clusters grows.

\begin{figure}[htp]
    \centering
    \begin{minipage}{0.47\textwidth}
        \centering
        \includegraphics[width=\linewidth]{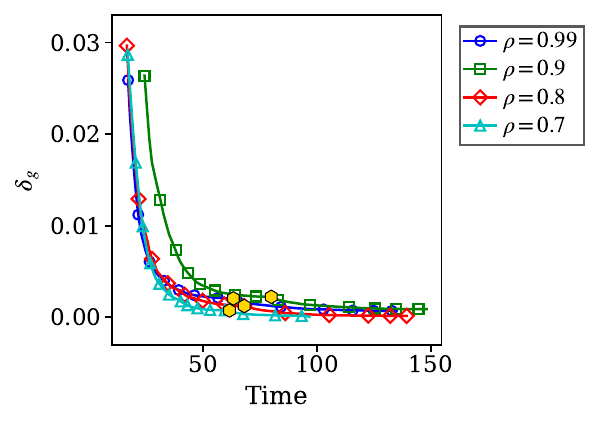}
        \subcaption{$K=3$}
    \end{minipage}
    \hfill
    \begin{minipage}{0.47\textwidth}
        \centering
        \includegraphics[width=\linewidth]{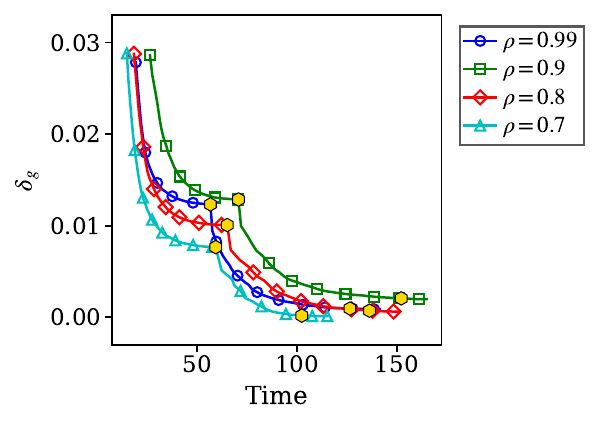}
        \subcaption{$K=4$}
    \end{minipage}
    \vspace{0.8em}

    \begin{minipage}{0.47\textwidth}
        \centering
        \includegraphics[width=\linewidth]{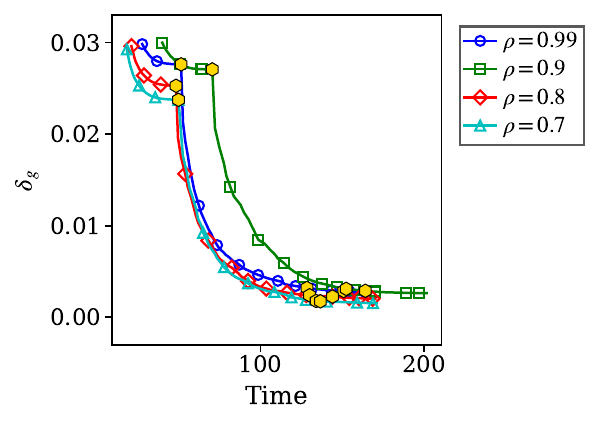}
        \subcaption{$K=5$}
    \end{minipage}
    \caption{Relative optimality gap $\delta_g$ over time for fair K-means clustering with $\alpha$-fair constraints for the {\tt Student Math} data set. Golden hexagons indicate the iterations in which the algorithm increases the separation parameter $t$.}
    \label{fig:alpha_fair_math}
\end{figure}

\begin{figure}[htp]
    \centering
    \begin{minipage}{0.47\textwidth}
        \centering
        \includegraphics[width=\linewidth]{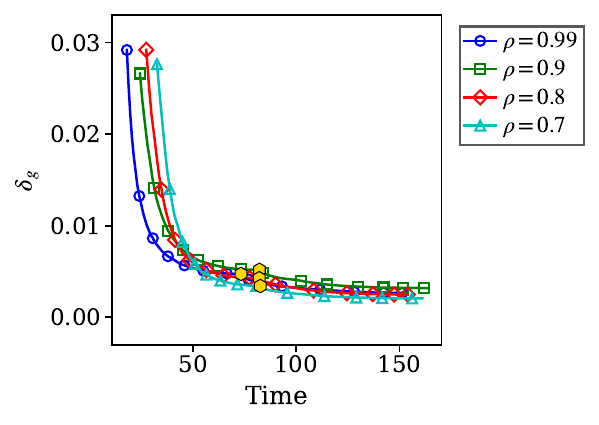}
        \subcaption{$K=3$}
    \end{minipage}
    \hfill
    \begin{minipage}{0.47\textwidth}
        \centering
        \includegraphics[width=\linewidth]{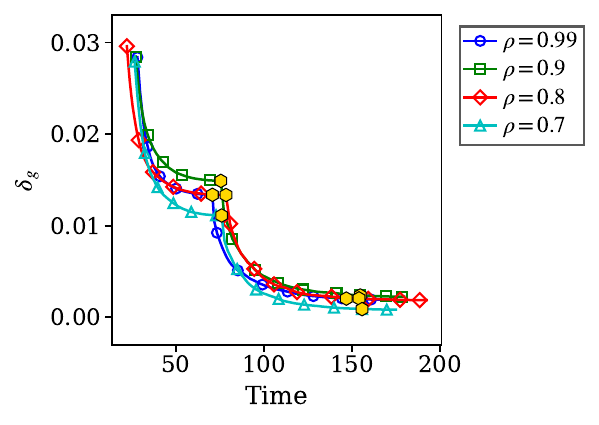}
        \subcaption{$K=4$}
    \end{minipage}
    \vspace{0.8em}

    \begin{minipage}{0.47\textwidth}
        \centering
        \includegraphics[width=\linewidth]{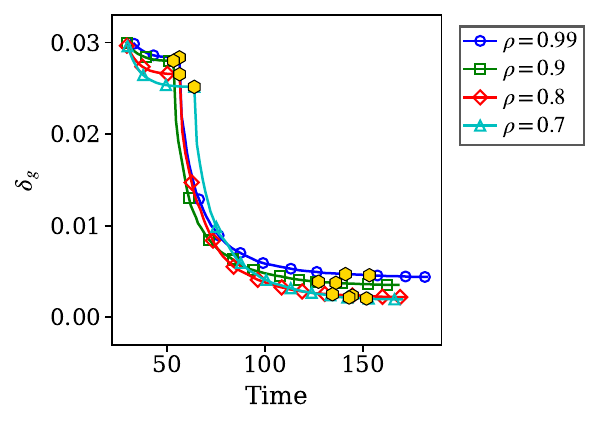}
        \subcaption{$K=5$}
    \end{minipage}
    \caption{Relative optimality gap $\delta_g$ over time for fair K-means clustering with $\tau$-fair constraints for the {\tt Student Math} data set. Golden hexagons indicate the iterations in which the algorithm increases the separation parameter $t$.}
    \label{fig:tau_fair_math}
\end{figure}

\begin{figure}[htp]
    \centering
    \begin{minipage}{0.47\textwidth}
        \centering
        \includegraphics[width=\linewidth]{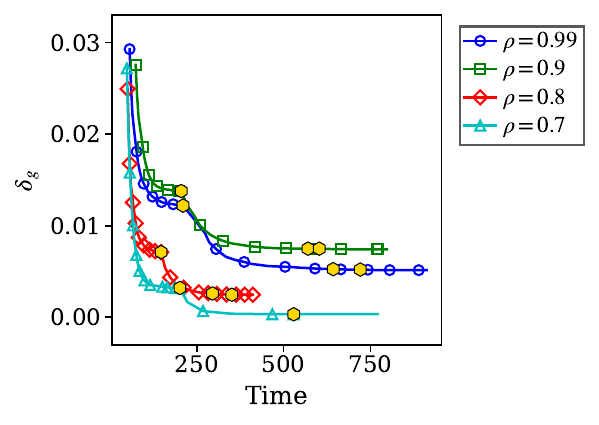}
        \subcaption{\texttt{Titanic 2}}
    \end{minipage}
    \hfill
    \begin{minipage}{0.47\textwidth}
        \centering
        \includegraphics[width=\linewidth]{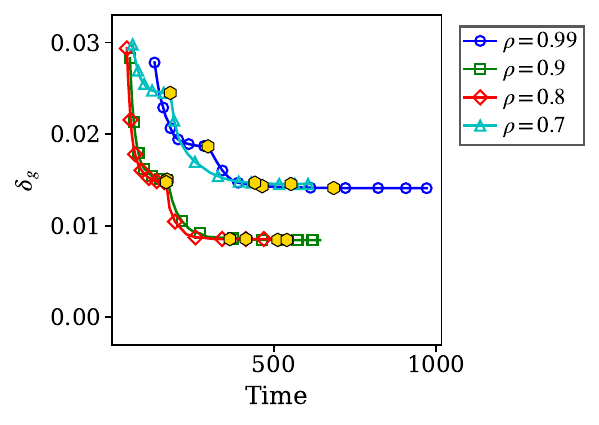}
        \subcaption{\texttt{Titanic 3}}
    \end{minipage}
    \vspace{0.8em}

    \begin{minipage}{0.47\textwidth}
        \centering
        \includegraphics[width=\linewidth]{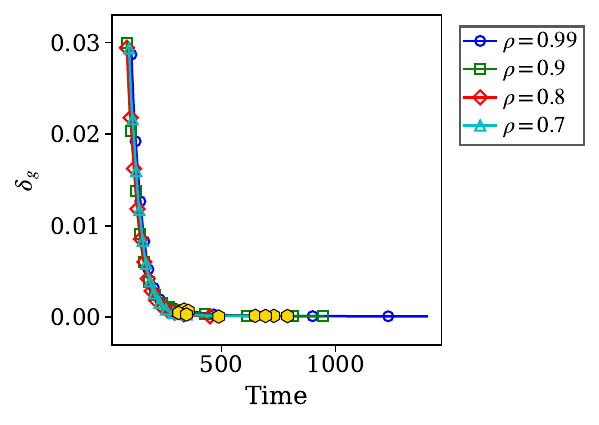}
        \subcaption{\texttt{Credit}}
    \end{minipage}
    \caption{Relative optimality gap $\delta_g$ over time for fair K-means clustering with $\alpha$-fair constraints and $K=5$ clusters for three data sets. Golden hexagons indicate iterations in which the algorithm increases the separation parameter $t$.}
    \label{fig:alpha_fair_selected}
\end{figure}

\begin{figure}[htp]
    \centering
    \begin{minipage}{0.47\textwidth}
        \centering
        \includegraphics[width=\linewidth]{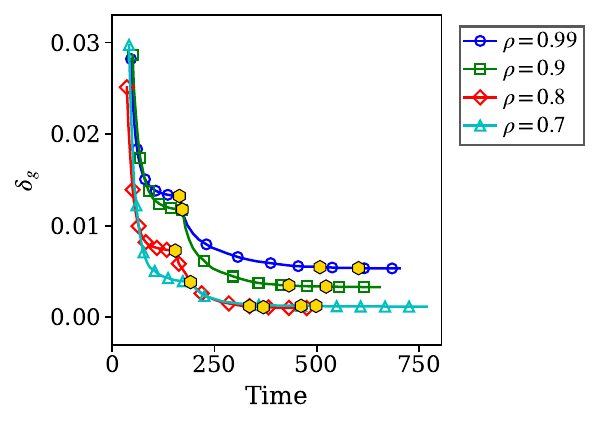}
        \subcaption{\texttt{Titanic 2}}
    \end{minipage}
    \hfill
    \begin{minipage}{0.47\textwidth}
        \centering
        \includegraphics[width=\linewidth]{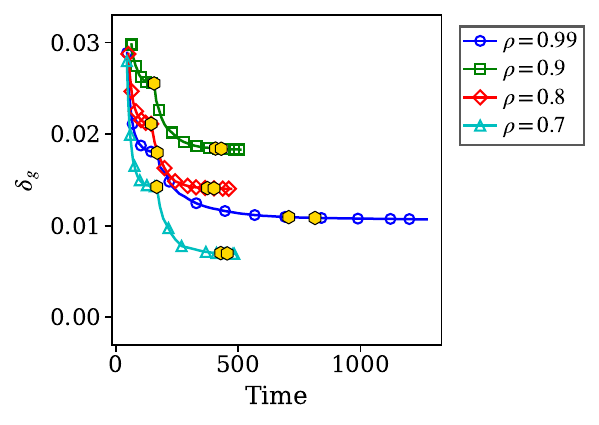}
        \subcaption{\texttt{Titanic 3}}
    \end{minipage}
    \vspace{0.8em}

    \begin{minipage}{0.47\textwidth}
        \centering
        \includegraphics[width=\linewidth]{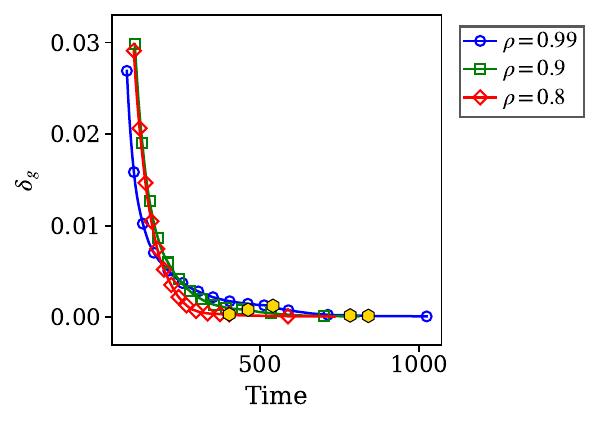}
        \subcaption{\texttt{Credit}}
    \end{minipage}
    \caption{Relative optimality gap $\delta_g$ over time for fair K-means clustering with $\tau$-fair constraints and $K=5$ clusters for three data sets. Golden hexagons indicate iterations in which the algorithm increases the separation parameter $t$.}
    \label{fig:tau_fair_selected}
\end{figure}
\clearpage

\end{document}